\documentclass{article}

\usepackage{lipsum}
\usepackage{amsfonts}
\usepackage{amssymb}
\usepackage{amsmath}
\usepackage{amsthm}
\usepackage{graphicx}
\usepackage{epstopdf}
\usepackage{algorithmic}
\usepackage{cleveref}
\usepackage[left=2cm,right=2cm,top=2cm,bottom=2cm]{geometry}
\usepackage[numbers,sort]{natbib}
\ifpdf
  \DeclareGraphicsExtensions{.eps,.pdf,.png,.jpg}
\else
  \DeclareGraphicsExtensions{.eps}
\fi

\newtheorem{theorem}{Theorem}
\newtheorem{proposition}[theorem]{Proposition}
\newtheorem{remark}[theorem]{Remark}
\newtheorem{lemma}[theorem]{Lemma}
\newtheorem{assumption}[theorem]{Assumption}

\usepackage{amsopn}

\DeclareMathOperator*{\esssup}{ess\,sup}

\DeclareMathOperator*{\argmin}{arg\,min}
\newcommand{\norm}[2]{\left\lVert #1\right\rVert_{#2}}

\renewcommand{\div}{\operatorname{div}}

\newcommand{\E}{\mathbb{E}}
\newcommand{\pP}{\mathbb{P}}
\newcommand{\R}{\mathbb{R}}

\newcommand{\D}{\textup{ d}}

\newcommand{\Kd}{\mathfrak{K}^d}
\newcommand{\dH}{d_\mathfrak{H}}
\newcommand{\gG}{\boldsymbol{\Upsilon}}
\newcommand{\Sabm}{\mathcal{S}_\Gamma}
\newcommand{\Kabm}{\mathcal{K}_\Gamma}
\begin{document}

\title{Two-norm discrepancy and convergence\\ 
of the stochastic gradient method\\
with application to shape optimization}


\author{Marc Dambrine%
\thanks{Universit\'e de Pau et des Pays de l'Adour,
IPRA-LMA, UMR CNRS 5142, Avenue de l'universit\'e, 64000 Pau, France
(\texttt{marc.dambrine@univ-pau.fr})}
\and
Caroline Geiersbach%
\thanks{Weierstrass Institute for Applied Analysis and Stochastics,
Mohrenstr.~39, 10117 Berlin, Germany (\texttt{geiersbach@wias-berlin.de})}
\and
Helmut Harbrecht%
\thanks{Department of Mathematics and Computer Science, 
University of Basel, Spiegelgasse 1, 4051 Basel, Switzerland
(\texttt{helmut.harbrecht@unibas.ch})}
}

\date{\today}

\maketitle

\begin{abstract}
The present article is dedicated to proving convergence 
of the stochastic gradient method in case of random shape 
optimization problems. To that end, we consider Bernoulli's 
exterior free boundary problem with a random interior boundary. 
We recast this problem into a shape optimization problem by means 
of the minimization of the expected Dirichlet energy. By restricting 
ourselves to the class of convex, sufficiently smooth domains 
of bounded curvature, the shape optimization problem becomes 
strongly convex with respect to an appropriate norm. Since this 
norm is weaker than the differentiability norm, we are confronted 
with the so-called two-norm discrepancy, a well-known phenomenon from 
optimal control. We therefore need to adapt the convergence theory 
of the stochastic gradient method to this specific setting
correspondingly. The theoretical findings are supported 
and validated by numerical experiments.
\end{abstract}

\section{Introduction}
Shape optimization under uncertainty is a topic of growing 
interest, see for example \cite{AD14,AD15,BCH20,Rumpf2009,DK24,MF18} 
and the references therein. The most common approach is the
minimization of the expectation of the shape functional. In
specific cases, this problem can be reformulated as a deterministic 
one, see e.g.~\cite{DDH15,DHP19}. However, this is not possible
in general, which makes the shape optimization algorithm quite 
costly. One popular approach for the minimization of the expectation 
is offered by the stochastic gradient method, which originated in 
\cite{Robbins1951} and has been used in recent years in the optimal 
control of partial differential equations involving uncertain inputs 
or parameters; see, e.g., \cite{Geiersbach2019a, martin2021complexity}. 
In the present article, we intend to verify the convergence of this 
method in case of \emph{Bernoulli's exterior free boundary 
problem\/} in case of a \emph{random interior boundary\/}. 
Bernoulli's exterior free boundary problem is an overdetermined 
boundary value problem for the Laplacian, where one has an 
inclusion with Dirichlet boundary condition and an exterior,
\emph{free boundary\/} with Dirichlet and Neumann boundary 
condition. This free boundary problem becomes random when we 
assume that the interior boundary is random.

The aforementioned random free boundary problem has already 
been considered in several articles in different settings by 
some of the authors of this article. Bernoulli's free boundary 
problem can be seen as a ``fruit fly'' of shape optimization,
see \cite{DHPP17,DambrineHarbrechtPuig,HP17}. In particular, 
much is known about existence and regularity of the solution 
to the free boundary problem in the deterministic setting, 
see e.g.~\cite{ACA,beurling,EP06,HenSha97,tepper} for some of 
such results. If we restrict ourselves, for example, to starlike 
domains and the interior boundary $\Sigma_1$ lies within the
boundary $\Sigma_2$, then the solution $\Gamma_1$ of the free
boundary problem for $\Sigma_1$ lies within the solution 
$\Gamma_2$ for $\Sigma_2$. This important monotonicity 
property helps to ensure well-posedness in case 
of randomness.

The mathematical formulation of Bernoulli's free boundary 
problem with a random interior boundary is given in 
Section~\ref{sct:problem}. We choose the Dirichlet boundary 
value problem as the state equation and reformulate the problem 
under consideration as a shape optimization problem for the 
state's Dirichlet energy. That way, a variational formulation 
of the desired Neumann boundary condition at the free boundary 
is derived. We then intend to minimize the mean of the energy
functional. Although the problem under consideration is well-posed 
in the deterministic case (see e.g.~\cite{EP06,eppler2007convergence}), 
the present random shape optimization problem is not, in 
general. This fact is motivated in Section~\ref{sec:analex} 
by an analytical example with circular boundaries. Indeed, 
the random interior boundary has to lie almost surely within 
some sufficiently narrow concentric annulus to ensure that
the sought free boundary does not intersect this annulus,
which would imply a degenerated situation. For the sake of 
simplicity, we will consider circular annuli throughout 
the rest of this article.

Section~\ref{sct:shape calculus} is then concerned with shape 
calculus in the case of the deterministic free boundary problem.
We provide the shape gradient and shape Hessian of the energy 
functional under consideration for general boundaries. Then,
we study the convexity of the shape optimization problem under
consideration. We are able to prove $H^{1/2}$-convexity for all
convex exterior boundaries that are sufficiently smooth, lie 
in a fixed annulus, and have a uniformly bounded curvature. This
is the one of the main results of our article and the key to 
verifying convergence of the iterates of stochastic gradient 
method. Additionally, to the best of our knowledge, global 
convexity has never been derived in shape optimization for 
a specific problem before.

In Section~\ref{sec:SGM-2norm}, we prove convergence of the 
stochastic gradient method in a novel setting, namely, one 
involving the so-called \emph{two-norm discrepancy}. The
two-norm discrepancy is a well-known phenomenon in optimal 
control and may occur in the infinite-dimensional setting 
since not all norms are equivalent; see in particular 
\cite{ioffe1979necessary,Casas2012a}. We note that, for 
shape optimization, convergence of approximation (deterministic) 
solutions with the two-norm discrepancy was already established 
in \cite{eppler2007convergence}. There, second-order sufficient 
conditions were used to ensure stability around a local optimum. 
Since we have in fact strong convexity for the free boundary 
problem, we are able to prove convergence to the unique minimum, 
even in the presence of uncertainty. This is a stronger result 
than can be expected in a typical shape optimization problem under 
uncertainty; we note that convergence of the stochastic gradient 
method was shown in the context of Riemannian manifolds in 
\cite{Geiersbach2021}. Due to the (geodesic) nonconvexity of the 
unconstrained problem studied there, one can at most expect that 
the corresponding Riemannian gradient vanishes in the limit. The 
main difficulty in the analysis here is that the convexity for 
the energy functional is with respect to a weaker space than 
the one to which the exterior boundaries belong. We provide a complete proof of convergence of iterates 
to the unique solution with respect to the weaker norm in the almost sure sense. We explain why the typical convergence rates in expectation cannot be derived in the function space setting due to the two-norm discrepancy. On the other hand, the discretized sequence will yield the expected rates for strongly convex functions.

Numerical experiments are presented in Section~\ref{sec:numerix} 
in order to validate the theoretical findings. For a random 
starlike interior boundary, we compute the solution of the present
random version of Bernoulli's free boundary problem. We observe 
very fast convergence towards the correct shape of the sought 
free boundary, which is a huge improvement over previously 
studied methods such as the use of sampling methods to compute 
the expected shape functional and its gradient or the direct 
computation of an appropriate expectation of the free boundary.
In all, we observe a rate of convergence with respect to 
necessary optimality condition that is inverse proportional 
to the square root of the number of iterations. The rate of 
convergence with respect to the objective function values is 
inverse proportional to the number of iterations, as predicted 
by the theory.

Throughout this article, for $D\subset\mathbb{R}^2$ being a 
sufficiently smooth domain, we denote the space of square
integrable functions by $L^2(D)$. For a nonnegative real 
number $s\ge 0$, the associated Sobolev spaces are labelled 
by $H^s(D)\subset L^2(D)$. Especially, there holds $H^0(D) 
= L^2(D)$. Moreover, when $s\ge 1/2$, the respective Sobolev 
spaces on the boundary $\partial D$ are defined as the traces 
$H^{s-1/2}(\partial D) := \gamma(H^s(D))$ of the Sobolev spaces 
$H^s(D)$. The set of $k$-times differentiable functions is 
denoted by $\mathcal{C}^k$ and $\mathcal{C}^{k,\alpha}$ denotes 
the set of functions in $\mathcal{C}^k$ whose $k$-th order 
partial derivatives are additionally $\alpha$-H\"older continuous.

\section{Problem setting}\label{sct:problem}
Let $(\Omega,\mathcal{F},\mathbb{P})$ be a complete 
probability space. In this article, we consider Bernoulli's 
free boundary problem when the interior boundary is random. 
The precise meaning of the random boundary will be specified 
later on. Given an event $\omega\in\Omega$, we are thus looking 
for an annular domain $D = D(\omega)\subset\mathbb{R}^2$ with 
interior boundary $\Sigma = \Sigma(\omega)$ and unknown deterministic 
exterior boundary $\Gamma$ such that the function $u = u(\omega)\in 
H^1\big(D(\omega)\big)$ satisfies the following Dirichlet boundary 
value problem for the Laplacian
\begin{equation}
\label{bvp}
    \begin{aligned}
    \Delta u &= 0 \quad \text{in} \quad D(\omega) ,\\
    u &= 1 \quad \text{on} \quad \Sigma(\omega),\\
    u &= 0 \quad \text{on} \quad \Gamma,
    \end{aligned}
\end{equation}
with the additional flux condition 
\begin{equation}
\label{flux}
    -\frac{\partial u}{\partial\boldsymbol{n}} = \lambda \quad \text{on} \quad \Gamma,
\end{equation}
where $\lambda > 0$ is a given constant. Here and
in the following, $\boldsymbol{n} = \boldsymbol{n}(\omega)$ 
denotes the exterior unit normal to $D(\omega)$. Note that the 
geometrical setup is illustrated in Figure~\ref{fig:setup}.

\begin{figure}
\begin{center}
\includegraphics[width=0.4\textwidth]{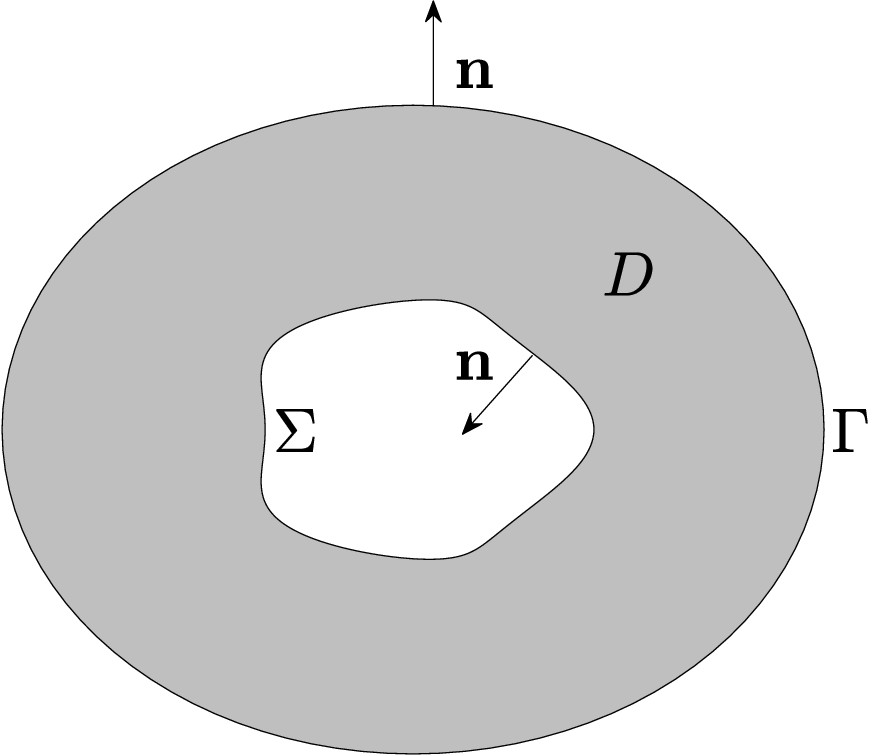}
\caption{\label{fig:setup}The geometrical setup:
The annular domain $D$ with the given interior boundary $\Sigma$
and the free exterior boundary $\Gamma$.}
\end{center}
\end{figure}

Bernoulli’s free boundary problem arises in many 
applications, for example in ideal fluid dynamics, 
optimal design, electrochemistry, or electrostatics.
Generally speaking, Bernoulli's free boundary problem is an 
overdetermined partial differential problem since, for fixed 
boundaries $\Sigma(\omega)$ and $\Gamma$, the unknown harmonic 
function $u = u(\omega)$ has to vanish at the outer boundary 
and also to satisfy a flux condition \eqref{flux}. However, 
it becomes solvable when the \emph{free boundary} $\Gamma$
is also considered as an unknown. We refer the reader to 
e.g.~\cite{ACA,Flucher,Friedman} and the references therein
for further details.

For any fixed realization $\Sigma(\omega)$ of the interior 
boundary, where $\omega\in\Omega$, it is well-known that a 
variational formulation for the sought boundary $\Gamma$
such that the overdetermined (deterministic) boundary value 
problem \eqref{bvp} and \eqref{flux} admits a solution is 
given by 
\begin{equation}\label{functional}
\underset{\Gamma \subset \R^2}{\textup{minimize}} \quad
    J(\Gamma,\Sigma(\omega)) = \int_{D(\omega)} \norm{\nabla u(\omega)}{2}^2 + \lambda^2 \D x
    = \int_{\Sigma(\omega)} \frac{\partial u}{\partial\boldsymbol{n}}\D s + \lambda^2 |D(\omega)|.
\end{equation}
where the state $u = u(\omega)$ is the solution to \eqref{bvp}.
Uniqueness and existence of solutions to this free 
boundary problem follows from the seminal work \cite{ACA}.

The free boundary defined as the minimizer of the shape 
optimization problem \eqref{functional} depends on the 
particular random event $\omega\in\Omega$. Therefore, 
in order to get a deterministic free boundary $\Gamma$ 
while accounting for all possibilities of $\Sigma(\cdot)$, 
we shall consider the minimization (with respect to the 
boundary $\Gamma$) of the expected functional
\begin{equation}\label{E[functional]}
\underset{\Gamma \subset \R^2}{\textup{minimize}} \quad
    \E[J(\Gamma,\Sigma(\cdot))] = \int_\Omega \int_{D(\omega)} 
    \norm{\nabla u(\omega)}{2}^2 + \lambda^2 \D x \D \pP(\omega)
\end{equation}
with the state $u = u(\omega)$ given by \eqref{bvp}.
Note that this is a free boundary problem, where the 
underlying domain $D$ is random. Although the pointwise
solution \eqref{functional} is well-defined, this does 
not necessarily hold true for \eqref{E[functional]}.
We motivate this fact in the next section.

\section{Analytical computations in the case of concentric annuli}
\label{sec:analex}
Calculations can be performed analytically 
if the interior boundary $\Sigma$ is a circle
around the origin with radius $r_\Sigma$. Then, 
due to symmetry, the free boundary $\Gamma$ 
will also be a circle around the origin with unknown 
radius $r_\Gamma$. 

Using polar coordinates and making the ansatz 
$u(r,\theta) = y(r)$, we find $\Delta u(r,\theta) = 
y''(r) + y'(r)/r$. Hence, the solution with respect 
to the prescribed Dirichlet boundary condition of 
\eqref{bvp} in the case of dimension two is given by
\[
  y(r) = \frac{\log\big(\frac{r}{r_\Gamma}\big)}
  	{\log\big(\frac{r_\Sigma}{r_\Gamma}\big)}.
\]
The desired Neumann boundary condition at 
the free boundary $r_\Gamma$ yields the equation 
\[
  -y'\big(r_\Gamma\big) = \frac{1}
  	{r_\Gamma\log\big(\frac{r_\Gamma}{r_\Sigma}\big)} = \lambda,
\]
which can be solved by means of Lambert's $W$-function:
\[
r_\Gamma = F(r_\Sigma) := \frac{1}{\lambda W\big(\frac{1}{\lambda r_\Sigma}\big)}.
\]
Let us recall that Lambert's $W$-function is the inverse of 
$x\mapsto x e^{x}$. It is a non-decreasing function on $(0,+\infty)$ 
which, however, provides a non-analytic expression.

Since the Neumann data of $u$ on the interior free
boundary $r_\Sigma$ are given by
\[
  -y'\big(r_\Sigma\big) = \frac{\partial u}{\partial \boldsymbol{n}}\Big|_{\Sigma}
    = \frac{1}{r_\Sigma\log\big(\frac{r_\Gamma}{r_\Sigma}\big)},
\]
we conclude
\[
  J(r_\Gamma,r_\Sigma) = \frac{2\pi}{\log\big(\frac{r_\Gamma}{r_\Sigma}\big)}
    + \pi \lambda^2 (r_\Gamma^2-r_\Sigma^2),
\]
compare \eqref{functional}. One readily verifies that 
this functional has indeed the unique minimizer 
$r_\Gamma = F(r_\Sigma)$.

Let us next consider the case where $r_\Sigma$ switches
randomly between $r_{\Sigma,1}$ and $r_{\Sigma,2}$ with
the probability $\mathbb{P}(r_\Sigma=r_{\Sigma,1}) = p$
and $\mathbb{P}(r_\Sigma=r_{\Sigma,2}) = 1-p$, where
$p\in [0,1]$. If we choose $r_{\Sigma,2}$ such that it
satisfies $r_{\Sigma,2} > F(r_{\Sigma,1})$, then we
obviously obtain the inequality chain
\begin{equation}\label{inequality}
 0 < r_{\Sigma,1} < F(r_{\Sigma,1}) < r_{\Sigma,2} < F(r_{\Sigma,2}) < \infty.
\end{equation}
The expected functional reads 
\[
  \E[J(r_\Gamma,r_\Sigma(\cdot))] 
    = \frac{2\pi p}{\log\big(\frac{r_\Gamma}{r_{\Sigma,1}}\big)}
    + \frac{2\pi(1-p)}{\log\big(\frac{r_\Gamma}{r_{\Sigma,2}}\big)}
    + \pi \lambda^2 \big(r_\Gamma^2-p r_{\Sigma,1}^2 - (1-p) r_{\Sigma,2}^2\big).
\]
Its unique minimizer is $r_\Gamma = 
F(r_{\Sigma,1})$ if $p = 1$ while it is 
$r_\Gamma = F(r_{\Sigma,2})$ if $p = 0$. In view 
of \eqref{inequality}, this means that $r_\Gamma$ 
has to cross $r_{\Sigma,2}$ during the transition 
from $p = 0$ to $p = 1$. However, this is impossible
since then the domain $D(\omega)$ is not well-defined 
anymore as $r_\Sigma(\omega)< r_\Gamma$ is violated.
Therefore, it is required to impose an inequality 
constraints to the sought boundary $r_\Gamma$, 
demanding that $r_\Gamma \ge \delta + r_\Sigma(\omega)$
for $\pP$-almost all $\omega\in\Omega$ (``almost surely'') and some $\delta > 0$.

The above observations motivate the assumption that
$\Sigma(\omega)$ lies inside some annulus such that
\begin{equation}\label{eq:interior_bound}
  B(0,\underline{r}_\Sigma)\subset\Sigma(\omega)\subset B(0,\overline{r}_\Sigma)
  \quad\text{almost surely}.
\end{equation}
Thus, it follows from \cite{beurling,tepper}
that the resulting free exterior boundary $\Gamma 
= \Gamma(\omega)$ satisfies
\[
  B(0,\underline{r}_\Gamma)\subset\Gamma(\omega)\subset B(0,\overline{r}_\Gamma)
  \quad\text{almost surely},
\]
provided that the interior domain surrounded by
$\Sigma(\omega)$ is starlike. In order to ensure 
well-posedness in our subsequent analysis, we restrict 
ourselves to starlike interior boundaries satisfying
\eqref{eq:interior_bound}, where $\underline{r}_\Sigma$
and $\overline{r}_\Sigma$ are such that 
$\overline{r}_\Sigma\le\underline{r}_\Gamma$.

\section{Properties of the objective with respect 
to shape variations}\label{sct:shape calculus}
We shall now focus on the particular case where the 
free boundary $\Gamma$ is the boundary of a convex domain. 
Indeed, it is known that the solution to Bernoulli's free 
boundary problem is starlike if the interior boundary is; see
\cite{tepper}. If the interior boundary is even convex, then the
exterior one is also convex \cite{HenSha97}. However, the
exterior boundary can also be convex although the interior is not.

Our main result in this section is the convexity 
of the objective with respect to an appropriate norm. Of
course, since shape spaces are not linear, convexity has to be
understood in terms of a parameterization of the boundary
$\Gamma$. Among the many ways to parameterize such a curve, we discuss
the case of the parameterization with respect to a point, then with
the support function. Since they do not recover exactly the same
geometric perturbations, the respective second-order derivative
has different properties. 

\subsection{Shape sensitivity analysis}
We first consider general geometries and perturbations 
and we compute the first and second order shape derivatives 
of the objective 
around a given boundary $\Gamma$ 
for general perturbations. To this end, we choose the interior
boundary $\Sigma$ arbitrary but fixed and suppress its explicit 
dependence in the objective $J$. Throughout this section,
the outer boundary $\Gamma$ is always such that it encloses the
interior boundary $\Sigma$ to ensure that the annular domain
$D$ in between is well-defined.

We first study the dependence with respect to the outer 
boundary $\Gamma$ and consider sufficiently regular 
deformation fields $\boldsymbol{V}$ that are defined
in the neighborhood of the exterior boundary $\Gamma$.
As we need regularity on the shapes, let us assume for 
convenience that $\Gamma$ is of class $\mathcal{C}^2$ and 
that the deformation field $\boldsymbol{V}$ has the same 
regularity. While the first-order derivative of the shape 
functional under consideration has already been calculated 
and used many times in the literature, the second-order
derivative has only been studied at a critical point for
stability analyses (see \cite{EP06}, for example). Here, 
we will study its expression for domains $D$ where the 
gradient does not vanish. For a
comprehensive introduction to shape calculus, we refer the 
reader to \cite{DEZ,HenrotPierre,SOZ}.

\begin{lemma}
\label{lemshapederivatives}
Let the boundaries $\Gamma$ and $\Sigma$ be of class 
$\mathcal{C}^2$. Then, the first- and second-order shape
derivatives for of the objective are given by
\begin{equation}
\label{shapegradient}
    D J(\Gamma)[\boldsymbol{V}] = \int_\Gamma \boldsymbol{V}_{\boldsymbol{n}}
    \left[ \lambda^2-\left(\frac{\partial u}{\partial\boldsymbol{n}}\right)^2\right]\D s
\end{equation}
and
\begin{equation}
    \label{shapehessian}
D^2 J(\Gamma)[\boldsymbol{V},\boldsymbol{V}] 
= \int_\Gamma \frac{\partial u'}{\partial\boldsymbol{n}} u' 
+ H  \lambda^2 \boldsymbol{V}_{\boldsymbol{n}}^2 
+ \left(\frac{\partial u}{\partial\boldsymbol{n}}\right)^2 
\boldsymbol{V}\cdot\nabla_{\boldsymbol{\tau}} \boldsymbol{V}_{\boldsymbol{n}} \D s,
\end{equation}
where we used the abbreviation $\boldsymbol{V}_{\boldsymbol{n}} 
= \boldsymbol{V}\cdot\boldsymbol{n}$   
and where 
$\nabla_{\boldsymbol{\tau}}$ 
denotes the surface gradient with
respect to the boundary $\Gamma$.
\end{lemma}

\begin{proof}
By usual arguments given in e.g.~\cite{DEZ,HenrotPierre,SOZ},
the solution $u$ of \eqref{bvp} has the following 
first- and second-order derivatives $u'$ and $u''$ 
that are characterized by differentiating the boundary 
condition on the boundary $\Gamma$:
\[
 \begin{aligned}
    \Delta u' &= 0 \quad \text{in} \quad D ,\\
    u' &= 0 \quad \text{on} \quad \Sigma,\\
    u' &= - \frac{\partial u}{\partial\boldsymbol{n}} \boldsymbol{V}_{\boldsymbol{n}}\quad \text{on} \quad \Gamma,
\end{aligned}
\]
and
\[
 \begin{aligned}
    \Delta u'' &= 0 \quad \text{in} \quad D ,\\
    u'' &= 0 \quad \text{on} \quad \Sigma,\\
    u'' &= - \left[\frac{\partial u'}{\partial\boldsymbol{n}}
    +\nabla u \cdot\boldsymbol{n}'\right]  \boldsymbol{V}_{\boldsymbol{n}}
    -\frac{\partial u}{\partial\boldsymbol{n}} \boldsymbol{V}\cdot \boldsymbol{n}'\\
    &=-\frac{\partial u'}{\partial\boldsymbol{n}}\boldsymbol{V}_{\boldsymbol{n}} 
    + \frac{\partial u}{\partial\boldsymbol{n}}\boldsymbol{V}
    \cdot\nabla_{\boldsymbol{\tau}} (\boldsymbol{V}_{\boldsymbol{n}})\quad \text{on} \quad \Gamma,
    \end{aligned}
\]
where we used that $\boldsymbol{n}'=-\nabla_{\boldsymbol\tau} (\boldsymbol{V}_{\boldsymbol{n}})$, 
hence $\nabla u\cdot\boldsymbol{n}'=0$. Therefore, we immediately
arrive at
\begin{equation*}
    D J(\Gamma)[\boldsymbol{V}] 
    = \int_\Gamma\boldsymbol{V}_{\boldsymbol{n}} \left[ \lambda^2
    -\left(\frac{\partial u}{\partial\boldsymbol{n}}\right)^2 \right] \D s
\end{equation*}
and 
\begin{equation*}
 \begin{aligned}
    D^2 J(\Gamma)[\boldsymbol{V},\boldsymbol{V}] 
    &= \int_\Gamma u'' \frac{\partial u}{\partial \boldsymbol{n}} 
    + H  \lambda^2 \boldsymbol{V}_{\boldsymbol{n}}^2 \D s \\
    &=  \int_\Gamma  \left( -\frac{\partial u'}{\partial\boldsymbol{n}}  
    \boldsymbol{V}_{\boldsymbol{n}} + \frac{\partial u}{\partial \boldsymbol{n}} 
    \boldsymbol{V}\cdot\nabla_{\boldsymbol{\tau}} (\boldsymbol{V}_{\boldsymbol{n}})\right)
    \frac{\partial u}{\partial \boldsymbol{n}} +H  \lambda^2 \boldsymbol{V}_{\boldsymbol{n}}^2 \D s\\
    &= \int_\Gamma \frac{\partial u'}{\partial \boldsymbol{n}} 
    \left( - \frac{\partial u}{\partial\boldsymbol{n}}\boldsymbol{V}_{\boldsymbol{n}} \right) 
    + H  \lambda^2 \boldsymbol{V}_{\boldsymbol{n}}^2 + \left(\frac{\partial u}{\partial \boldsymbol{n}}\right)^2 
    \boldsymbol{V}\cdot\nabla_{\boldsymbol{\tau}} (\boldsymbol{V}_{\boldsymbol{n}}) \D s\\
    &=\int_\Gamma \frac{\partial u'}{\partial\boldsymbol{n}} u' 
    + H  \lambda^2 \boldsymbol{V}_{\boldsymbol{n}}^2  \D s+ \int_\Gamma  
    \left(\frac{\partial u}{\partial \boldsymbol{n}}\right)^2 
    \boldsymbol{V}\cdot\nabla_{\boldsymbol{\tau}} \boldsymbol{V}_{\boldsymbol{n}}  \D s.
\end{aligned}  
\end{equation*}
\end{proof} 

\subsection{On the sign of the shape Hessian}
In order to the study the sign of the shape Hessian, we split it 
into two terms
$$
  D^2 J(\Gamma)[\boldsymbol{V},\boldsymbol{V}]
   = I_1(\boldsymbol{V}) + I_2(\boldsymbol{V}), 
$$
where we set
$$
I_1(\boldsymbol{V})=\int_\Gamma \frac{\partial u'}{\partial\boldsymbol{n}} u' 
+ H\lambda^2 \boldsymbol{V}_{\boldsymbol{n}}^2 \D s
\quad\text{and}\quad
I_2(\boldsymbol{V})=\int_\Gamma\left(\frac{\partial u}{\partial\boldsymbol{n}}\right)^2 
\boldsymbol{V}\cdot\nabla_{\boldsymbol\tau} \boldsymbol{V}_{\boldsymbol{n}} \D s.
$$
Notice that the term
$\boldsymbol{V}\cdot\nabla_{\boldsymbol{\tau}}
\boldsymbol{V}_{\boldsymbol{n}}$ appears in the former expressions 
as expected by the structure theorems of second order shape 
derivatives (see \cite[Theorem 5-9-2, page 220]{HenrotPierre} and 
\cite[Theorem 2-1]{DambrineLamboley}) since we are not at the optimum 
and because we do not restrict ourselves to normal perturbations. 
According to the structure of the shape Hessian, the first term 
$I_1$ is a quadratic form in $\boldsymbol{V}_{\boldsymbol{n}}$. 
The second term $I_2$ is a remainder of the shape gradient. It 
is bilinear in $\boldsymbol{V}$ but also involves tangential 
derivatives. Consequently, finding the sign of the shape
Hessian requires studying the two terms separately.
\subsubsection{On the sign of $\boldsymbol{I_1}$}
Recall the we assume that $\Gamma$ is the boundary of 
a convex set. Hence, its curvature $H$ is nonnegative and 
we obtain after integration by parts 
$$
I_1(\boldsymbol{V})=\int_\Gamma \frac{\partial u'}{\partial\boldsymbol{n}} u' 
+ H  \lambda^2 \boldsymbol{V}_{\boldsymbol{n}}^2  \D s
\geq \int_\Gamma \frac{\partial u'}{\partial\boldsymbol{n}} u'  \D s
= \int_D \lVert \nabla u'\rVert_2^2  \D x>0,
$$
i.e., the integral $I_1$ is clearly positive. Since $u'=0$ 
on the component $\Sigma$ of the boundary of $D$, we then 
get by Poincaré's inequality and the trace theorem that 
\begin{equation}
    \label{estimation:coercivite:I_1}
I_1(\boldsymbol{V})= \int_D \|\nabla u'\|_2^2 \D x \geq C_P(D) \|u'\|_{H^{1/2}(\Gamma)}^2 ,  
\end{equation}
where $C_P(D)$ is the Poincaré constant of the domain 
$D$ with homogeneous boundary condition on $\Sigma$. 
Abbreviating $\partial_{\boldsymbol{n}} u = (\partial u)
/(\partial \boldsymbol{n})$, an immediate first lower bound 
is thus
$$
I_1(\boldsymbol{V})\geq C_P(D)\|u'\|_{L^2(\Gamma)}^2=C_P(D)  
\| (\partial_{\boldsymbol{n}} u) \boldsymbol{V}_{\boldsymbol{n}}\|_{L^2(\Gamma)}^2
\geq C_P(D)  (\inf_\Gamma \partial_{\boldsymbol{n}} u)^2 
\| \boldsymbol{V}_{\boldsymbol{n}}\|_{L^2(\Gamma)}^2.
$$
Here, we have used the strong maximum principle to ensure that 
$\inf_\Gamma \partial_{\boldsymbol{n}} u>0$. 

In fact, we can have a more precise lower bound in the Sobolev norm 
$H^{1/2}(\Gamma)$. To that end, we use the following lemma to estimate 
the Sobolev norm of the product $u'=- (\partial_{\boldsymbol{n}} u) 
\boldsymbol{V}_{\boldsymbol{n}}$ from below.

\begin{lemma}
If $f\in H^{1/2}(\Gamma)$ and $g\in H^{1/2+\epsilon}(\Gamma)$, 
then there exists some $C>0$ such that
\begin{equation}
\label{estimate:product:H^1/2(surface)}
\|fg\|_{H^{1/2}(\Gamma)} 
\leq  C \|f\|_{H^{1/2}(\Gamma)} \ \|g\|_{H^{1/2+\epsilon}(\Gamma)}.
\end{equation}
If there exists some $a>0$ such that $g\geq a$ on $\Gamma$, then
there exists some $C>0$ such that
\begin{equation}
    \label{minoration:produit:H^1/2(surface)}
\|f\|_{H^{1/2}(\Gamma)} \leq C \|fg\|_{H^{1/2}(\Gamma)}
    \ \|1/g\|_{H^{1/2+\epsilon}(\Gamma)}. 
\end{equation}
\end{lemma}

\begin{proof}
The key ingredient is the following product estimate 
in Sobolev spaces taken from \cite[Lemma 7-2]{BehzadanHolst}: 
If $F\in H^1(D)$ and $G\in H^{1+\epsilon}(D)$ for some 
$\epsilon>0$, then the product satisfies $FG \in H^1(D)$
and we have
\begin{equation}
\label{estimate:product:H^1(volume)}
\|FG\|_{H^1(D)} \leq C \|F\|_{H^1(D)} \|G\|_{H^{1+\epsilon}(D)}.
\end{equation}

We now translate this estimate to the trace space on 
the boundary $\Gamma$. To this end, set $f\in H^{1/2}(\Gamma)$ 
and $g\in H^{1/2+\epsilon}(\Gamma)$. Let $F$ and $G$ be harmonic 
extensions of $f$ and $g$ to $D$ so that by Dirichlet's principle 
$\|f\|_{H^{1/2}(\Gamma)} = \|F\|_{H^1(D)}$ and 
$\|g\|_{H^{1/2+\epsilon}(\Gamma)}=\|G\|_{H^{1+\epsilon}(D)}$,
respectively. Then, in view of \eqref{estimate:product:H^1(volume)}, 
we first get
$$
\|FG\|_{H^1(D)} \leq C \|F\|_{H^1(D)}\ \|G\|_{H^{1+\epsilon}(D)} 
=  C \|f\|_{H^{1/2}(\Gamma)} \ \|g\|_{H^{1+\epsilon}(\Gamma)},
$$
and then by the definition of the trace norm
\begin{equation*}
\|fg\|_{H^{1/2}(\Gamma)} 
= \inf_{\substack{\phi \in H^1(D)\\ \phi =fg \text{ on }\Gamma}} 
\| \phi\|_{H^1(D)} \leq \|FG\|_{H^1(D)} 
\leq  C \|f\|_{H^{1/2}(\Gamma)} \ \|g\|_{H^{1/2+\epsilon}(\Gamma)}.
\end{equation*}

Assume now that $g$ satisfies the additional property 
that there exists some $a>0$ such that $g\geq a$ on $\Gamma$. 
Consider the function $I_a$ defined on $(0,+\infty)
\to\mathbb{R}$ by $I_a(t)=1/a$ if $t\leq a$ and by 
$I_a(t)=1/t$ otherwise. Obviously, this is a bounded 
Lipschitz function. We notice that $1/g = I_a \circ g$ and 
hence $1/g$ belongs to $H^{1/2+\epsilon}(\Gamma)$. As a 
consequence, since $f = (fg)\ (1/g)$, we obtain by the 
product estimate \eqref{estimate:product:H^1/2(surface)} that
\begin{equation*}
\|f\|_{H^{1/2}(\Gamma)} \leq C \|fg\|_{H^{1/2}(\Gamma)}
    \ \|1/g\|_{H^{1/2+\epsilon}(\Gamma)}. 
\end{equation*}
\end{proof}

With the help of this lemma, we obtain the following result.

\begin{lemma}
There exists a constant $C>0$ depending on $\Gamma$ such that
$$
I_1(\boldsymbol{V})\geq C\| \boldsymbol{V}_{\boldsymbol{n}}\|_{H^{1/2}(\Gamma)}^2.
$$
\end{lemma}

\begin{proof}

Under our regularity assumptions on the boundaries $\Sigma$ and 
$\Gamma$, we can apply \eqref{minoration:produit:H^1/2(surface)} for
$f=\boldsymbol{V}_{\boldsymbol{n}}$ and $g=\partial_{\boldsymbol n} u$ so that $u'=-fg$. 
The lower bound on $g$ comes from the strong maximum principle 
and the compactness of $\Gamma$. Hence, we have 
$$
\|\boldsymbol{V}_{\boldsymbol{n}}\|_{H^{1/2}(\Gamma)} 
\leq C \|\boldsymbol{V}_{\boldsymbol{n}} \partial_{\boldsymbol{n}} u\|_{H^{1/2}(\Gamma)}    
\ \|1/(\partial_{\boldsymbol{n}} u)\|_{H^{1/2+\epsilon}(\Gamma)}. 
$$
The claim then follows from \eqref{estimation:coercivite:I_1}.
\end{proof}

\subsubsection{On the sign of $\boldsymbol{I_2}$}
The sign of the second term 
$$
I_2(\boldsymbol{V}) 
= \int_\Gamma\left(\frac{\partial u}{\partial\boldsymbol{n}}\right)^2 
\boldsymbol{V}\cdot\nabla_{\boldsymbol\tau} \boldsymbol{V}_{\boldsymbol{n}} \D s
$$
is less clear since it has the sign of the purely geometric term 
$\boldsymbol{V}\cdot\nabla_{\boldsymbol{\tau}} \boldsymbol{V}_{\boldsymbol{n}}$. 
Indeed, the sign of that term and hence of $I_2$ depends on the specific 
class of perturbations under consideration. 

The natural parameterization of convex domains is the one using 
support functions. We restrict ourselves to perturbations 
of a convex domain that preserve convexity. We then check that 
in this situation $I_2$ takes only nonnegative values.

Recently, shape calculus for convex domains based on the 
Minkowski sum and therefore on support functions was developed 
in \cite{Boulkhemair,BoulkhemairChakib}). For our purposes, 
however, it suffices to use simpler tools. To this end, let 
us recall the definition of the support function and its main 
properties. Convex sets $K\subset\mathbb{R}^d$ are parameterized 
by their support function $h_K$ defined on $\R^d$ by 
$$
h_K(x)=\sup \{ x\cdot y \mid y \in K \}.
$$
The monotonicity property 
$
K_1\subset K_2\Rightarrow h_{K_1}\leq h_{K_2}
$ 
is clear from this definition. In particular, for nonnegative real numbers 
$a<b$, we have 
$$
B(0,a)\subset K \subset B(0,b)\Rightarrow a\leq h_K\leq b.
$$
The support function of a convex set is homogeneous of degree 
one and hence can be restricted to the unit sphere $\mathbb{S}^{d-1}$ 
without loss of generality. To simplify notation, we are still abusively 
calling this restriction $h_K$. 

Let us introduce the parameterization mapping $\gG$ defined 
over the set $\Kd$ of convex domains in $\R^d$ by
$$
 \gG : \Kd \rightarrow \mathcal{C}^0(\mathbb{S}^{d-1}), \quad K\mapsto h_K.
$$
A crucial property is the isometric connection between the Hausdorff 
distance and the $L^{\infty}$-norm 
on $\mathcal{C}^0(\mathbb{S}^{d-1},\R)$: 
for all $K_1, K_2 \in \Kd$, 
\begin{equation}
\label{relation:cruciale}
\dH(K_1,K_2) = \|h_{K_1}-h_{K_2}\|_{L^{\infty}(\mathbb{S}^{d-1})}.
\end{equation} 
Reconstructing a convex set from a support function can be performed using the envelope operator  
$$
\mathcal{E}\colon \mathcal{C}^1(\mathbb{S}^{d-1},\R)\to 
\mathcal{C}^1(\mathbb{S}^{d-1},\R^d),\quad
  h\mapsto\mathcal{E}[h],
$$
defined for all $x \in\mathbb{S}^{d-1}$ by
$$
\mathcal{E}[h](x) = h(x) x + \nabla_{\boldsymbol{\tau}} h(x).
$$
This operator allows to reconstruct a convex set whose 
restricted support is $h$. Notice that this point 
has been investigated in the works of Antunes and Bogosel
\cite{AntunesBogosel,Bogosel}. 

In the planar case, one gets simply a periodic function
$ h\colon [0,2\pi]\to\R$ and a parameterization of a set whose 
support function $h$ is 
$$
\mathcal{E}[h]: 
\theta \mapsto  h(\theta) \boldsymbol{e}_r(\theta)
+ h'(\theta) \boldsymbol{e}_\theta(\theta). 
$$
Here and in the following, $\boldsymbol{e}_r(\theta) = 
(\cos\theta,\sin\theta)$ denotes the radial direction and 
$\boldsymbol{e}_\theta(\theta) := \boldsymbol{e}_r(\theta)' 
= (-\sin\theta,\cos\theta)\perp\boldsymbol{e}_r$. Thus, we get 
\begin{equation}\label{eq:derivative_supp_fct}
\begin{aligned}
\mathcal{E}[h]'(\theta)&=h'(\theta) \boldsymbol{e}_r(\theta) 
+ h(\theta)\boldsymbol{e}_\theta(\theta)
+ h''(\theta) \boldsymbol{e}_\theta(\theta) 
- h'(\theta) \boldsymbol{e}_r(\theta) \\
&= \big(h(\theta)+h''(\theta)\big)\boldsymbol{e}_\theta(\theta).
\end{aligned}
\end{equation}
Therefore, the unit tangent vector $\boldsymbol{\tau}$ 
at $\mathcal{E}[h](\theta)$ is $\boldsymbol{e}_\theta(\theta)$ 
and the unit outward normal vector $\boldsymbol{n}$ is 
then $\boldsymbol{e}_r(\theta)$. A perturbation $q$ of 
the support function generates the support function $h+tq$ 
for $|t|$ sufficiently small and thus the parameterization 
$$
\mathcal{E}[h+q](\theta) = (h+tq)(\theta)\boldsymbol{e}_r(\theta)  
+ (h+tq)'(\theta)\boldsymbol{e}_\theta(\theta).
$$
Therefore, the deformation field is
$$
\boldsymbol{V} = \frac{\D}{\D t}\mathcal{E}[h+tq]
= q \boldsymbol{e}_r + q' \boldsymbol{e}_\theta 
$$
and we conclude $\boldsymbol{V}_{\boldsymbol n} = q$. Notice that 
this expression also makes sense in the neighborhood of the curve 
$\Gamma$. We can next compute directly the gradient and observe 
that it is tangent to the curve 
$$
\nabla\boldsymbol{V}_{\boldsymbol{n}}
= \cfrac{q'}{\sqrt{h^2+(h')^2}}\,\boldsymbol{e}_\theta 
= \nabla_{\boldsymbol{\tau}}\boldsymbol{V}_{\boldsymbol{n}}, 
$$
which implies
$$
\boldsymbol{V}\cdot\nabla_{\boldsymbol\tau}\boldsymbol{V}_{\boldsymbol{n}} 
= \cfrac{(q')^2}{\sqrt{h^2+(h')^2}} \ge 0.
$$
Since $h>0$ (indeed, we need a disk in the inner domain here 
that is uniformly greater than zero), we have herewith 
shown that $I_2[\boldsymbol{q}]\geq 0$.

\begin{remark}{Under radial deformations of convex domains, $I_2$ has no sign.} 
In the case of starlike domains, the outer boundary $\Gamma$ can
be parameterized by $\gamma\boldsymbol{e}_r$, 
where $\gamma:
[0,2\pi] \rightarrow (0,+\infty)$ denotes the radial function
and $\boldsymbol{e}_r = (\cos\theta,\sin\theta)$ is the radial 
direction. Then, the unit tangent vector $\boldsymbol{\tau}$ 
and the unit outward normal vector $\boldsymbol{n}$ are given 
by the formulae
$$
\boldsymbol{\tau} = \cfrac{1}{\sqrt{\gamma^2+(\gamma')^2}} 
\left(\gamma' \boldsymbol{e}_r +\gamma\boldsymbol{e}_\theta\right) 
\quad \text{and}\quad 
\boldsymbol{n} = \cfrac{1}{\sqrt{\gamma^2+(\gamma')^2}} 
\left(\gamma \boldsymbol{e}_r -\gamma'\boldsymbol{e}_\theta\right).
$$
Thus, the normal component of any boundary deformation field of 
the type $\boldsymbol{V} = \varphi\boldsymbol{e}_r$ is
$$
\boldsymbol{V}_{\boldsymbol{n}} = \boldsymbol{V} \cdot \boldsymbol{n} 
= \cfrac{\gamma \varphi}{\sqrt{\gamma^2+(\gamma')^2}}.
$$
Hence, we find 
$$
\nabla (\boldsymbol{V} \cdot \boldsymbol{n}) 
= \cfrac{1}{\gamma} \left(\cfrac{\gamma \varphi}
{\sqrt{\gamma^2+(\gamma')^2}}\right)'\boldsymbol{e}_\theta
$$
and thus
$$
\nabla_{\boldsymbol\tau} (\boldsymbol{V} \cdot \boldsymbol{n}) 
= \cfrac{1}{\gamma^2+(\gamma')^2} \left(\cfrac{\gamma \varphi}
{\sqrt{\gamma^2+(\gamma')^2}}\right)' 
\left(\gamma' \boldsymbol{e}_r +\gamma\boldsymbol{e}_\theta\right).
$$
Consequently, the term $\boldsymbol{V}\cdot\nabla_{\boldsymbol\tau}
\boldsymbol{V}_{\boldsymbol{n}}$ is given by
\begin{align*}
\boldsymbol{V}\cdot \nabla_{\boldsymbol{\tau}} (\boldsymbol{V} \cdot 
\boldsymbol{n})&=\cfrac{\gamma'\varphi}{\gamma^2+(\gamma')^2} 
\left(\cfrac{\gamma \varphi}{\sqrt{\gamma^2+(\gamma')^2}}\right)'\\
&= \cfrac{(\gamma'\varphi)^2}{\big(\gamma^2+(\gamma')^2\big)^{3/2}}
+ \cfrac{(\gamma^2)'}{2\sqrt{\gamma^2+(\gamma')^2}}
\left(\cfrac{\varphi^2}{\gamma^2+(\gamma')^2}\right)'.
\end{align*}
Obviously, this previous expression has no sign 
since $\varphi$ is arbitrary. 
\end{remark}

\subsubsection{Restricting the objective to a class 
of domains to get a strongly convex one}
The previous investigation of the shape Hessian motivates 
the study of the free boundary problem under consideration
in the class of convex domains, parameterized by means
of the support function. To this end, we identify
the boundary $\Gamma$ with its support function and
set $\mathcal{J}(h) = J\big(\mathcal{E}[h]\big)$. In view of the above results 
and translating them in terms of support function (see 
\cite{Boulkhemair,BoulkhemairChakib}), we have proven that 
$$
    D \mathcal{J}(h)[q]
    = \int_\Gamma q(\theta_{\boldsymbol{n}}) \left[ \lambda^2
    -\left(\frac{\partial u}{\partial\boldsymbol{n}}\right)^2 \right] \D s
$$
while for the shape Hessian one gets
\begin{equation}
\label{local:lower:bound:hessian}
   D^2 \mathcal{J}(h)[q] 
    \geq C(h) \|q(\theta_{\boldsymbol{n}})\|^2_{H^{1/2}(\Gamma)} 
\end{equation}
by combining the estimates on $I_1$ and 
$I_2$. Herein, $\theta_{\boldsymbol{n}}$ 
is the angle $\theta\in [0,2\pi]$ that is imposed 
by the normal vector, i.e., $\boldsymbol{n} = (\cos 
\theta_{\boldsymbol{n}},\sin\theta_{\boldsymbol{n}})$.
This is a convexity result but a weak one. Its main 
weakness is its non-uniformity with respect to the 
design variable. Moreover, it is not formulated in 
a differentiation norm.

\subsection{Uniform lower bounds of the Hessian}
We shall next study when there exists a uniform bound that
is independent of $\Gamma$. To that end, we proceed with 
the following strategy. First, we introduce a parameterization 
of the family of domains under consideration
by restricting ourselves to starlike boundaries $\Sigma$ and 
$\Gamma$. The boundary value problem is first transported to a 
fixed annulus, resulting in a family of parameterized problems 
on that annulus. The local inversion theorem is employed to 
demonstrate the regularity of the map associating the boundaries 
to the solution of the parameterized boundary value problem. 
Subsequently, the boundaries are restricted to a compact 
context for the parameterization, allowing us to obtain uniform bounds.  

In that spirit, for given positive numbers $\alpha\in (0,1/2)$, 
$0<r_{\underline{\Gamma}}<r_{\overline{\Gamma}}<M_\Gamma$, we 
consider the class $\Sabm$ of periodic functions defined 
on $[0,2\pi]$ by 
\begin{equation}
\label{def:class:support:functions}
\Sabm =\{ h\in \mathcal{C}_{per}^{3,2\alpha} \mid \forall \theta\in[0,2\pi],\  
r_{\underline{\Gamma}}\leq h(\theta)\leq r_{\overline{\Gamma}},\ 
(h+h'')(\theta)\ge 0, \text{ and } \|h\|_{\mathcal{C}^{3,2\alpha}} \leq M_\Gamma\}. 
\end{equation}
This is a compact subset of $\mathcal{C}_{per}^{3,\alpha}$ that 
parameterizes through the support function the class
\begin{equation}
\label{def:class:concex:sets}
\Kabm =\{ K \subset \mathbb{R}^2 \mid \exists h \in\Sabm,\  
    \Gamma=\partial K=\mathcal{E}[h] \},
\end{equation}
of convex subsets of $\mathbb{R}^2$ with a $\mathcal{C}_{per}^{2,2\alpha}$ 
boundary between the two concentric circles of radii $r_{\underline{\Gamma}}$
and $r_{\overline{\Gamma}}$, respectively. Note that the sets $\Sabm$
and $\Kabm$ are convex and closed, respectively, as for any pair of
functions $h_1, h_2\in\Sabm$, the convex combination $\lambda h_1
+ (1-\lambda)h_2$ is also a member of the class $\Sabm$ for all $\lambda\in(0,1)$.

With this notation at hand, we are now in the position to state 
the main result of this section. 

\begin{proposition} \label{uniform:weak:convexity:on:Sabm}
Given positive numbers $\alpha\in (0,1/2)$, $0<r_{\underline{\Gamma}}
<r_{\overline{\Gamma}}<M_\Gamma$, there exists a positive number $C$ 
depending on $r_{\underline{\Gamma}}$, $r_{\overline{\Gamma}}$, and 
$M_\Gamma$ such that for all $h\in\Sabm$, 
\begin{equation}
\label{uniform:lower:bound:hessian}
   D^2 \mathcal{J}(h)[q] 
    \geq C \|q(\theta_{\boldsymbol{n}})\|^2_{H^{1/2}(\Gamma)} .
\end{equation}
\end{proposition}

To prove Proposition \ref{uniform:weak:convexity:on:Sabm}, we check the 
uniform behavior of each constant in the successive inequalities we used. 
These are
\begin{itemize}
    \item the Poincaré inequality in \eqref{estimation:coercivite:I_1}. 
    The uniform bound follows from the geometric bounds of $\Gamma$.
    \item the product inequality \eqref{estimate:product:H^1(volume)}. 
    The uniform bound follows from the existence of uniform (with 
    respect to $D \in \Kabm$) extension operators for the 
    Sobolev spaces $H^1(D)$ and  $H^{1+\epsilon}(D)$ 
    to the whole $H^1(\mathbb{R}^2)$ and $H^{1+\epsilon}(\mathbb{R}^2)$, 
    induced by the upper bound for $\|h\|_{\mathcal{C}^{3,2\alpha}}$.  
    \item the equivalence between the trace norm and the intrinsic 
    Sobolev norm for fractional Sobolev spaces on a boundary for 
    the upper bound of $1/g$ by composition. Gagliardo has shown 
    in \cite{Gagliardo} that the two different norms on $H^{1/2}$ 
    are equivalent if the domains are uniformly Lipschitz.
    \item finally, the lower bound for the normal derivative 
    $\partial_{\boldsymbol{n}} u$. 
\end{itemize}
The latter item is less standard, hence we shall elaborate on it. The 
main difficulty we face here is to get a uniform lower bound of the
normal derivative. Clearly, it is nonnegative thanks to the maximum 
principle. Nevertheless, by its own, this argument cannot provide a 
uniform lower bound. We need an additional ingredient: continuity 
and compactness with respect to the inner and outer boundaries. 

We transform the boundary value problem with variable boundaries
to a boundary value problem with fixed boundary but variable coefficients.
The boundaries $\Sigma$ and $\Gamma$ are parameterized by $\sigma(\theta)
{\bf e}_r$ and $\gamma(\theta){\bf e}_r$, respectively. Here, 
\begin{equation}\label{eq:radii}
 \sigma: [0,2\pi]\to[\underline{r}_\Sigma,\overline{r}_\Sigma],
 \quad\gamma: [0,2\pi]\to[\underline{r}_\Gamma,\overline{r}_\Gamma]
\end{equation}
denote the associated radial functions of the interior and 
exterior boundaries and ${\bf e}_r = \big(\cos(\theta),\sin(\theta)\big)$ 
is the radial vector. Consider the annulus $\circledcirc$ with bounds 
$\overline{r}_\Sigma < \underline{r}_\Gamma$. Then, the map
\[
 \Phi:\circledcirc\to\Omega,\quad
 (r,\theta)\mapsto\bigg[\frac{r-\overline{r}_\Sigma}{\underline{r}_\Gamma-\overline{r}_\Sigma}
 \gamma(\theta) + \frac{\underline{r}_\Gamma-r}{\underline{r}_\Gamma-\overline{r}_\Sigma}
 \sigma(\theta)\bigg]{\bf e}_r
\]
maps the annulus $\circledcirc$ one-to-one to the annular domain 
$\Omega$ described by the boundaries $\Sigma$ and $\Gamma$.

\begin{lemma}
The singular values of the Jacobian $\Phi'(r,\theta)$
are uniformly bounded from above and below for all 
$(r,\theta)$ from the annulus $\circledcirc$ provided
that the parameterizations $\gamma$ and $\sigma$ satisfy
\eqref{eq:radii} with uniformly bounded derivatives.
\end{lemma}

\begin{proof}
We shall compute the Jacobian of the map $\Phi$. With ${\bf e}_\theta 
= \big(-\sin(\theta),\cos(\theta)\big)$, we find 
\begin{align*}
 \Phi'(r,\theta) &= \frac{\partial\Phi(r,\theta)}{\partial r}{\bf e}_r^\top
    + \frac{1}{r}\frac{\partial\Phi(r,\theta)}{\partial \theta}{\bf e}_\theta^\top \\
  &= \frac{1}{\underline{r}_\Gamma-\overline{r}_\Sigma}
  [\gamma(\theta)-\sigma(\theta)]{\bf e}_r{\bf e}_r^\top
  + \frac{1}{r}\bigg[\frac{r-\overline{r}_\Sigma}{\underline{r}_\Gamma-\overline{r}_\Sigma}
 \gamma'(\theta) + \frac{\underline{r}_\Gamma-r}{\underline{r}_\Gamma-\overline{r}_\Sigma}
 \sigma'(\theta)\bigg]{\bf e}_r{\bf e}_\theta^\top\\
 &\qquad+ \frac{1}{r}\bigg[\frac{r-\overline{r}_\Sigma}{\underline{r}_\Gamma-\overline{r}_\Sigma}
 \gamma(\theta) + \frac{\underline{r}_\Gamma-r}{\underline{r}_\Gamma-\overline{r}_\Sigma}
 \sigma(\theta)\bigg]{\bf e}_\theta{\bf e}_\theta^\top.
\end{align*}
The Jacobian $\Phi'(r,\theta)$ is hence triangular
with diagonal entries
\[
  a(r,\theta) := \frac{\gamma(\theta)-\sigma(\theta)}
    {\underline{r}_\Gamma-\overline{r}_\Sigma},
  \quad b(r,\theta) := \frac{1}{r}\bigg[\frac{r-\overline{r}_\Sigma}
  {\underline{r}_\Gamma-\overline{r}_\Sigma}\gamma(\theta) 
  + \frac{\underline{r}_\Gamma-r}
  {\underline{r}_\Gamma-\overline{r}_\Sigma}\sigma(\theta)\bigg].
\]
and the off-diagonal entry
\[
  c(r,\theta) := \frac{1}{r}\bigg[\frac{r-\overline{r}_\Sigma}
  {\underline{r}_\Gamma-\overline{r}_\Sigma}\gamma'(\theta) 
  + \frac{\underline{r}_\Gamma-r}
  {\underline{r}_\Gamma-\overline{r}_\Sigma}\sigma'(\theta)\bigg].
\]
In view of
\[
  \underline{r}_\Sigma<\overline{r}_\Sigma\le r
  \le\underline{r}_\Gamma<\overline{r}_\Gamma
\]
and
\[
  0 < \underline{r}_{\Gamma}-\overline{r}_\Sigma 
  \le \gamma(\theta)-\sigma(\theta) 
  \le \overline{r}_{\Gamma}-\underline{r}_\Sigma < \infty
\]
for all $(r,\theta)$ from the annulus $\circledcirc$, 
the diagonal entries $a(r,\theta)$ and $b(r,\theta)$ 
are uniformly bounded from above and below for all 
annular domains $\Omega$ with starlike boundaries 
such that \eqref{eq:radii} holds. In addition, the 
modulus $|c(r,\theta)|$ of the off-diagonal entry is uniformly
bounded from above if $\gamma'(\theta)$ and $\sigma'(\theta)$ 
are. Straightforward calculation verifies that consequently 
the singular values of the Jacobian $\Phi'(r,\theta)$ are 
uniformly bounded from above and below.
\end{proof}

The boundary value problem \eqref{bvp} posed on $\Omega$
can be transformed to a boundary value problem in $\circledcirc$
by using the map $\Phi$. There holds 
\begin{equation}
\label{bvptransported}
    \begin{aligned}
    \div(A\nabla u) &= 0 \quad \text{in} \quad \circledcirc ,\\
    u &= 1 \quad \text{on} \quad \lVert x\rVert_2=\underline{r}_\Gamma,\\
    u &= 0 \quad \text{on} \quad \lVert x\rVert_2=\overline{r}_\Sigma ,
    \end{aligned}
\end{equation}
with the diffusion matrix $A$ is given by
\[
  A(r,\theta) = \big(\Phi'(r,\theta)\big)^{-1}
    \big(\Phi'(r,\theta)\big)^{-\top}\det\big(\Phi'(r,\theta)\big).
\]
Therefore, the diffusion matrix $A$ depends on the 
parameterizations $\gamma$ and $\sigma$ of the inner 
and outer boundaries and on their first order derivatives. 
If both $\gamma$ and $\sigma$ are in the H\"older class 
$\mathcal{C}^{2,\alpha}$ for some given $\alpha\in(0,1)$, 
then the diffusion matrix $A$ is $\mathcal{C}^{1,\alpha}$-smooth.

By the assumption $h\in \Sabm$, there exists a positive real 
number $M > 0$ such that the diffusion matrix distribution 
belongs to the subset $K_M$ of $ \mathcal{C}^{1,\alpha}
(\R_{sym}^{2\times 2})$ 
given by
$$
K_M=\left\{
A\in\mathcal{C}^{1,\alpha}(\R_{sym}^{2\times 2}) 
\text{ with } \frac{1}{M} \lVert\xi\rVert_2^2 \leq A\xi\cdot\xi\leq M\lVert\xi\rVert_2^2 
\text{ and } \|DA\|_{\infty} \leq M \right\}.
$$

Using standard arguments (local inverse theorem and a priori H\"older 
bounds up to the boundary, see \cite[Theorem 6-6, page 98]{GilbargTrudinger}, 
the solution map $A \mapsto v_A$ of the boundary value problem 
\eqref{bvptransported} is smooth and in particular continuous 
on the subset $K_M$ of $ \mathcal{C}^{1,\alpha}(\R_{sym}^{2\times 2})$.
The Neumann boundary data of the transformed solution are computed by 
$A\nabla u\cdot\boldsymbol{n}$.
In view of the uniform bounds 
on the singular values of $\Phi'$, we conclude the 
$L^\infty(0,2\pi)$-bound
\[
  0< c\le\partial_{\boldsymbol{n}} u(\theta) \le C<\infty
\]
for the Neumann derivative at the boundary $\Gamma$.

\subsection{Different interior boundaries}
So far, we have shown uniform bounds of the 
shape Hessian in the case of a fixed interior boundary
$\Sigma$. In particular, all the constants in these
bounds depend on this specific $\Sigma$. However, compactness arguments similar to those used for $\Gamma$ also apply
to $\Sigma$. Hence, if we assume that $\Sigma$
is starlike with periodic $\mathcal{C}^{2,2\alpha}$-smooth 
parameterization $\sigma(\theta){\bf e}_r$ such that 
$r_{\underline\Sigma}\le\sigma(\theta)\le 
r_{\overline\Sigma}$ for all $\theta\in [0,2\pi]$
and $\|\sigma\|_{\mathcal{C}^{2,2\alpha}}\le M_\Sigma$, then the 
uniform bounds still hold. In other words, we shall 
consider parameterizations from the set
\[
\mathcal{S}_\Sigma := \{\sigma\in\mathcal{C}_{per}^{2,2\alpha} |\ \forall \theta\in[0,2\pi],\  
r_{\underline{\Sigma}}\leq \sigma(\theta)\leq r_{\overline{\Sigma}}
\text{ and } \|\sigma\|_{\mathcal{C}^{2,2\alpha}} \leq M_\Sigma\}. 
\]
It can easily be shown that this set is also convex and closed. 
It is then sufficient to repeat the argument from the previous 
subsection (transporting the state equation from $D$ onto 
$\circledcirc$ and using explicit formulas as a function of 
the parameterization of the interior boundary) to show that 
the state and then the objective are continuous with 
respect to the interior boundary, and that the coercivity 
constant of the shape Hessian with respect to the support 
function of the outer boundary can be chosen to be uniform 
with respect to the interior boundary.

Consequently, the functional $\mathcal{J}$ defined on 
$\mathcal{K}_\Gamma \times \mathcal{S}_\Sigma$ by 
$$
\mathcal{J}(h,\sigma)= \int_{D} \norm{\nabla u}{2}^2 + \lambda^2 \D x , 
$$ 
where the boundary of $D$ has two connected components, 
the interior one being parameterized by the distance $\sigma$ 
to the origin and the outer one through the support function $h$. 
This functional $\mathcal{J}$ is continuous and there exists a 
constant $c_E>0$ such that for all $\sigma\in\mathcal{S}_\Sigma$ 
and all $h\in\mathcal{K}_\Gamma$ one has the estimate
\begin{equation}
\label{eq:coercivity-parametrized}
D^2_{h,h}\mathcal{J}(h,\sigma)[q,q]\geq c_E \|q\|_{H^{1/2}}^2
\quad\forall q\in H_{per}^{1/2}([0,2\pi]).
\end{equation}

To conclude this section, we define the random boundary $\Sigma$ 
as the image by the polar parameterization of a vector-valued random variable 
$\sigma \in L^\infty(\Omega,\mathcal{S}_\Sigma)$, where the latter set is 
comprised of (strongly) $\pP$-measurable functions $\sigma$ from $\Omega$ to 
the closed and convex set $\mathcal{S}_\Sigma$ and satisfying $\esssup_{\omega 
\in \Omega} \lVert \sigma(\omega)\rVert_{\mathcal{C}^{2,2\alpha}}<\infty$ 
for all $\sigma \in \mathcal{S}_\Sigma$. 
The continuity of 
the map $\mathcal{J}$ on $\mathcal{K}_\Gamma\times \mathcal{S}_\Sigma$ 
implies the measurability of the map $\omega \mapsto 
\mathcal{J}(h,\sigma(\omega))$. 

\section{On the stochastic gradient method with two-norm discrepancy}
\label{sec:SGM-2norm}

In this section, we prove convergence of the projected
stochastic gradient method for an abstract setting involving the two-norm discrepancy. This classical method dating back to Robbins and Monro \cite{Robbins1951} involves randomly sampling the otherwise intractable gradient and has been well-investigated in the literature. For the function space setting without this discrepancy, the stochastic gradient method and its variants have already been analyzed; see \cite{Goldstein1988, Yin1990, Culioli1990, Barty2007} and more recent contributions in the context of PDE-constrained optimization under uncertainty \cite{Geiersbach2019a, Geiersbach2020, Geiersbach2021c, Geiersbach2023, martin2021complexity}.
The setting we present in \cref{sec:abstract-setting} is adapted from \cite{eppler2007convergence}, where convergence of a deterministic Ritz--Galerkin-type method was proven. We show that this framework fits the free boundary problem investigated in the previous sections, where it was established that the energy functional is coercive in a weaker space than where it is continuous. In \cref{sec:SGM}, we present the method, which involves a modification of the typical projected stochastic gradient iteration whereby a stochastic gradient is computed on the weaker space and a projection is performed onto the stronger space. We provide a complete proof of almost sure convergence of iterates 
to the unique solution with respect to the weaker norm.

\subsection{Abstract setting}
\label{sec:abstract-setting}
In this section, we summarize our numerical approach to solving 
the free boundary problem \eqref{bvp}--\eqref{flux}. Let $X\subset H$
be two Hilbert spaces, which are dense in $L^2$, 
endowed with the 
inner products $( \cdot,\cdot)_X$ and $(\cdot, \cdot)_H$, respectively, 
and corresponding norms $\lVert \cdot\rVert_X$ and $\lVert \cdot \rVert_H$. 
The respective dual spaces are denoted by $X^*$ and $H^*$, which yields 
the Gelfand chain $X \subset H\subset L^2 \subset H^*\subset X^*$.
A ball centered at $r$ in a space $X'$ is denoted by $B_{\delta}^{X'}(r) 
= \{ h \in X' \mid \lVert h-r\rVert_{X'}< \delta \}$. We assume that 
$X \subset H$ with continuous embedding and that $X_{ad} \subset X$ 
is a bounded, closed, convex, and nonempty admissible set. 

Let $(\Omega, \mathcal{F}, \pP)$ be a complete probability space and 
$\xi \colon \Omega \to\Xi$ be a random function mapping to a (real) 
complete separable metric space $\Xi$. Now, consider the problem 
\begin{equation}\label{eq:SAproblem}
\underset{h \in X_{ad}}{\textup{minimize}} \quad \{j(h) = \E[\mathcal{J}(h,\xi)] \},
\end{equation}
where 
we assume that $j \colon O \subset X \rightarrow \R$ is twice continuously differentiable on an open set $O \supset X_{ad}$. A point $h^*$ is said to be a \textit{local solution} of \eqref{eq:SAproblem} in $X'$ if $j(h^*) \leq j(h)$ for all $h \in X_{ad} \cap B_{\delta}^{X'}(h^*)$ for some $\delta >0$. A necessary condition for $h^*\in X_{ad}$ to be a local solution (in $X$) to 
\eqref{eq:SAproblem} is given by the variational inequality 
\begin{equation}\label{eq:OC}
D j(h^*)[h-h^*] \geq 0 \quad \forall h \in X_{ad}.
\end{equation}
In the event that $j$ is convex on $X_{ad}$, this condition is 
also sufficient and $h^*$ is even a global solution. Notice that since $H$ is the weaker space, any local solution of \eqref{eq:SAproblem} in $H$ is also a local solution in $X$; i.e., local solutions in $H$ also satisfy the condition \eqref{eq:OC}. Our strategy of handling the two-norm discrepancy will be to show that our method converges in $H$ to a local solution.

In our application, it is only possible to show continuity and 
coercivity of $D^2 j$ on the weaker space $H$. In the abstract 
setting, this translates to the following assumption, which is
motivated by \cite{eppler2007convergence}. 
\begin{assumption}\label{assumption:j}
The functional $j \colon O\to\R$ is twice continuously 
differentiable. We assume that for every $h \in X_{ad},$ 
the second derivative $D^2 j(h)\in 
\mathcal{L}(X,X^*)$ extends continuously to a bilinear 
form on $H\times H$, i.e., there exists $C_S >0$ such that 
\begin{equation}\label{eq:assumption-continuity-bilinear-form}
    |D^2 j(h)[q_1,q_2]| \leq C_S \lVert q_1 \rVert_H \lVert q_2 \rVert_H  
    \ \forall q_1, q_2 \in H.
\end{equation}
Additionally, we assume that for there exists $c_E >0$ such that 
the following strong coercivity condition is satisfied for every $h \in X_{ad}$:
\begin{equation}
    \label{eq:strong-coercivity} 
    D^2j(h)[q,q] \ge c_E \lVert q\rVert_H^2 
    \quad \forall q \in H.
\end{equation}
\end{assumption}

We note that this assumption does not require that the 
objective $j$ extends continuously from $X$ to $H$, as 
this is not satisfied by our problem. Twice 
continuous differentiability of $j$  provides Lipschitz 
continuity of the corresponding gradient with respect to the 
$(X^*,X)$-duality. We note that since the second  derivative is continuously extendable to a bilinear form on $H \times H$, we have Lipschitizianity 
of $Dj$ on the weaker space, as was shown in \cite{eppler2007convergence}. 

The inequality \eqref{eq:strong-coercivity} implies strong convexity with respect to $H$. Indeed, a Taylor expansion around any point $h' \in X_{ad}$ gives \[j(h'+q)=j(h')+Dj(h')[q]+\tfrac{1}{2}D^2j(\eta)[q,q]\] with $\eta$ being a point on the segment between $h'$ and $h'+q$. Therefore,
 \begin{equation}
 \label{eq:strong-convexity}
 j(h'+q) - j(h') \geq Dj(h')[q] + \frac{c_E}{2} \lVert q \rVert^2_H
 \end{equation}
 provided that $h'+q \in X_{ad}$. If $h'=h^*$ is a (local) optimum, we have using \eqref{eq:OC} that
 \begin{equation}
  \label{eq:strong-convexity-2}
 j(h)-j(h^*) \geq \frac{c_E}{2} \lVert h-h^*\rVert_H^2 \quad \forall h \in X_{ad}.
 \end{equation}
In particular, $h^*$ is a strict minimizer in $X$ since $j(h) > j(h^*)$ for $h\neq h^*.$
Since the coercivity space $H$ differs from the stronger space $X$ over which $j$ is continuous, we cannot expect to have strong convexity with respect to the stronger space. On the other hand, it is possible in certain cases to show that a strict minimizer $h^*$ in $X$ is also one with respect to the weaker topology; see \cite{Casas2012a}.

\paragraph{The free boundary problem} We shall now look in more detail at how the free boundary problem fits into this framework. The space $H$ is the fractional 
Sobolev space $H^{1/2}_{per}([0,2\pi])$, the energy space of 
the shape Hessian. The smaller space $X$ is the Sobolev space $H^4_{per}
([0,2\pi])$, since it is continuously embedded into $C^{3,2\alpha}_{per}([0,2\pi])$
for all $\alpha\in (0,1/4)$ by the Sobolev embedding theorem. The 
set $X_{ad}$ is 
$$
X_{ad} =\{ h\in X \mid \forall \theta\in[0,2\pi]: 
r_{\underline{\Gamma}}\leq h(\theta)\leq r_{\overline{\Gamma}},\ 
(h+h'')(\theta)\ge 0, \text{ and } \|h\|_X \leq M_\Gamma\}. 
$$
Concerning \Cref{assumption:j}, we recall that the inner boundary is modeled as random with $\xi = \sigma$. The objective is therefore 
\[j(h) = \E[\mathcal{J}(h,\sigma(\cdot)]. \]
The continuity estimate \eqref{eq:assumption-continuity-bilinear-form}
is obvious by \eqref{shapehessian} as $\mathcal{K}_\Gamma\times 
\mathcal{S}_\Sigma$ is a compact set, see also \cite{EP06}. 
For \eqref{eq:strong-coercivity}, we have 
from \eqref{eq:coercivity-parametrized} that
\begin{equation}
\label{eq:almost-sure-lower-bound}
D^2 \mathcal{J}(h,\sigma(\cdot))[q,q] \ge c_E \lVert q\rVert_H^2 
    \quad \forall q \in H \text{ a.s. }
\end{equation}
Using the fact that $\esssup_{\omega \in \Omega} \lVert \sigma(\omega) \rVert_{\mathcal{C}^{2,2\alpha}}<\infty$, it is straightforward to argue that $D^2j(h) = \E[D^2\mathcal{J}(h,\sigma(\cdot))]$. Applying the expectation on both sides of \eqref{eq:almost-sure-lower-bound} yields  \eqref{eq:strong-coercivity}.

\subsection{Stochastic gradient method}
\label{sec:SGM}

Let $\pi_{X_{ad}} \colon X \rightarrow X_{ad}$ be a projection onto the set $X_{ad}$, defined by
\[
\pi_{X_{ad}}(h) = \argmin_{w \in X_{ad}} \, \lVert h-w\rVert_X, 
\]
which is well-defined and single-valued since $X_{ad} \subset X$ is assumed to be nonempty, closed, and convex.   We assume that it is possible to compute an approximation of the gradient in the form of a \textit{stochastic gradient} $G\colon X \times \Xi \rightarrow X$, which is defined as the (parameterized) Riesz representative of the mapping $D\mathcal{J}(\cdot,\hat{\xi}) \colon X \rightarrow X^*$, i.e., we have for every $\hat{\xi} \in \Xi$ that
\begin{equation}
\label{eq:grad_projection}
(G(h,\hat{\xi}),q)_X = \langle D\mathcal{J}(h,\hat{\xi}),q\rangle_{X^*,X} \quad \forall (h,q) \in X_{ad}\times X.
\end{equation}
We use the notation $\nabla j$ for the  gradient of $j$ in $X$, i.e., $(\nabla j(h),q)_{X} = \langle Dj(h),q\rangle_{X^*,X}$, where $h,q \in X$. 
The projected stochastic gradient method relies on a recursion of the form
\begin{equation}
\label{eq:SGM-recursion}
h_{n+1} := \pi_{X_{ad}}(h_n - t_n G(h_n,\xi_n)),
\end{equation}
where $h_1 \in X_{ad}$ and $\xi_n$ is randomly sampled from the law $\mathbb{P}\circ\xi^{-1}$ independently of previous samples $\xi_1, \dots, \xi_{n-1}.$ We require that the step sizes given in \eqref{eq:SGM-recursion} satisfy the Robbins--Monro rule from the original paper \cite{Robbins1951} on stochastic approximation:
 \begin{equation}\label{eq:RobbinsMonroStepSize}
t_n \geq 0, \quad \sum_{n=1}^\infty t_n = \infty, \quad  \sum_{n=1}^\infty t_n^2 < \infty.
 \end{equation}

We will show that the recursion \eqref{eq:SGM-recursion} with the step sizes \eqref{eq:RobbinsMonroStepSize} converges using similar arguments to those used in \cite{Geiersbach2019a}. Note that the convergence result there applies to problems formulated over a single Hilbert space (without the two-norm discrepancy) and so cannot be immediately used for our setting. Here, we also work with assumptions that are verifiable for our application. For completeness, therefore, we provide a proof. 

First, we recall some concepts that will be of use in the proof. A filtration is a sequence $\{ \mathcal{F}_n\}$ of sub-$\sigma$-algebras of $\mathcal{F}$ such that {$\mathcal{F}_1 \subset \mathcal{F}_2 \subset \cdots \subset \mathcal{F}.$} 
Given a Banach space $Y$, we define a discrete $Y$-valued stochastic process as a collection of $Y$-valued random variables indexed by $n$, in other words, the set $\{ \beta_n \mid \Omega \rightarrow Y \mid n \in \mathbb{N}\}.$ 
The stochastic process is said to be adapted to a filtration $\{ \mathcal{F}_n \}$ if and only if $\beta_n$ is $\mathcal{F}_n$-measurable for all $n$. Suppose $\mathcal{B}(Y)$ denotes the set of Borel sets of $Y$.
The natural filtration is the filtration generated by the sequence $\{\beta_n\}$ and is given by $\mathcal{F}_n = \sigma(\{\beta_1, \dots ,\beta_n\}) = \{ \beta_i^{-1}(B)\mid B \in \mathcal{B}(Y), i = 1, \dots, n\}$.
If for an event $F \in \mathcal{F}$ we have that $\mathbb{P}(F) = 1$, we say $F$ occurs almost surely (a.s.).
For an integrable random variable $\beta\colon \Omega \rightarrow \R$, the conditional expectation is denoted by $\E[\beta | \mathcal{F}_n]$, which is itself a random variable that is $\mathcal{F}_n$-measurable and which satisfies $\int_A \E[\beta | \mathcal{F}_n](\omega) \D \pP(\omega) = \int_A \beta(\omega) \D \pP(\omega)$ for all $A \in \mathcal{F}_n$.  

To demonstrate convergence, we will apply the following lemma.

\begin{lemma}[Robbins--Siegmund \cite{Robbins1971}]\label{lemma:convegence_nonnegative_supermartingales}
Assume that $\{\mathcal{F}_n\}$ is a filtration and $v_n$,
$a_n$, $b_n$, $c_n$ nonnegative random variables adapted to $\{\mathcal{F}_n\}$. If
$$\mathbb{E} [v_{n+1}| \mathcal{F}_n ] \leq v_n(1 + a_n) + b_n - c_n \mbox{ a.s. }$$
and $\sum_{n=1}^{\infty} a_n < \infty, \, \sum_{n=1}^{\infty} b_n < \infty$ a.s.,
then with probability one, $\lbrace v_n \rbrace $ is
convergent and $\sum_{n=1}^{\infty} c_n < \infty.$
\end{lemma}
We will also need the following result.
\begin{proposition}[\cite{Geiersbach2019a}]
\label{prop:technicalproposition}
Let $\{\tau_n \}$ be a nonnegative deterministic sequence and $\{ \beta_n\}$ a nonnegative random sequence in $\mathbb{R}$ adapted to $\{\mathcal{F}_n$\}. Assume that
$\sum_{n=1}^\infty \tau_n = \infty$ and  $\mathbb{E} [\sum_{n=1}^\infty \tau_n \beta_n] < \infty$. Moreover assume that 
$\beta_{n} - \mathbb{E}[\beta_{n+1}|\mathcal{F}_n] \leq \gamma \tau_n$ a.s.~for all $n$ and some $\gamma > 0$. Then
$$\beta_n  \hbox{ converges to } 0  \hbox{ a.s. }$$
\end{proposition}

To ensure convergence of \eqref{eq:SGM-recursion}, we make the following assumptions, which is a slight modification of those used in~\cite[Theorem 3.6]{Geiersbach2019a}. Since in our application, $X_{ad}$ is bounded, we can reasonably impose a uniform bound on the second moment term in Assumption~\ref{assumption:gradient} (iii) instead of the growth condition used in \cite{Geiersbach2019a}.
\begin{assumption}\label{assumption:gradient}
Let $\{\mathcal{F}_n \}$ be an increasing sequence of $\sigma$-algebras.  For each $n$, there exist $b_n$, $w_n$ with
$$b_n = \E[G(h_n,\xi_n) | \mathcal{F}_n] - \nabla j(h_n), \quad w_n = G(h_n, \xi_n) - \E[G(h_n, \xi_n) | \mathcal{F}_n],$$ 
which satisfy the following assumptions:\\ (i) $ h_n$ and $b_n $ are $\mathcal{F}_n$-measurable; (ii) for $K_n:=\esssup_{\omega \in \Omega} \lVert b_n(\omega)\rVert_X$ we have that \mbox{$ \sum_{n=1}^\infty t_n K_n< \infty$} and \mbox{$\sup_n K_n < \infty$}; (iii) there exists $M\geq 0$ such that $\E[\lVert G(h,\xi)\rVert_X^2] \leq M$ for all $r \in X_{ad}$.
\end{assumption}
The sequence $b_n$ is a bias term that can be neglected if the stochastic gradient can be computed exactly for $\xi_n$. 
The following result follows using similar arguments to those made in\cite[Theorem 3.6]{Geiersbach2019a}. The main difference is that strong convergence occurs with respect to the weaker norm $H$, even though the iterates belong to $X$. Here, some arguments are simplified since we assume that $X_{ad}$ is bounded. 

\begin{theorem}\label{lemma:convergenceofPSG}
Suppose that Assumption~\ref{assumption:j} and Assumption~\ref{assumption:gradient} hold. If the sequence of step sizes satisfy \eqref{eq:RobbinsMonroStepSize}, then for iterates defined by the recursion \eqref{eq:SGM-recursion}, we have
\begin{enumerate}
 \item $\{j(h_n) \}$ converges a.s.\,and $\lim_{n \rightarrow \infty} j(h_n) = j(h^*)$,
 \item $\{h_n\}$ almost surely converges weakly in $X$ to $h^*$, and
 \item $\{h_n \}$ almost surely converges strongly in $H$ to $h^*$. 
\end{enumerate}
\end{theorem}
\begin{proof}
Recall that since $j$ is strongly convex with respect to $H$, a unique minimum $h^*\in X_{ad}$ exists. We note that $\pi_{X_{ad}}(h^*) = h^*$. Now, with $g_n:=G(h_n,\xi_n)$, we use the nonexpansivity of the projection to obtain
\begin{equation}
\label{eq:first-steps-convergence}
\begin{aligned}
    \lVert h_{n+1} -h^* \lVert_X^2 &= \lVert \pi_{X_{ad}} (h_n - t_n g_n) - \pi_{X_{ad}}(h^*)\rVert_X^2\\
    & \leq \lVert h_n - t_n g_n - h^*\lVert_X^2\\
    &=\lVert h_n - h^* \lVert_X^2 - 2t_n ( g_n, h_n-h^*)_X + t_n^2 \lVert  g_n\rVert_X^2.
\end{aligned}
\end{equation}

Due to (strong) convexity and the fact that $h_n, h^* \in X_{ad}$, we have 
\[ j(h^*) - j(h_n) \geq Dj(h_n)[h^*-h_n] + \frac{c_E}{2}\lVert h^*-h_n\rVert_H^2 \geq Dj(h_n)[h^*-h_n]\]
so that $-(j(h_n) - j(h^*)) \geq -D j(h_n)[h_n-h^*]=- (\nabla j(h_n), h_n-h^*)_X$. Moreover, optimality of $h^*$ gives \eqref{eq:OC}.
Taking the conditional expectation on both sides of \eqref{eq:first-steps-convergence} and applying Cauchy--Schwarz for the bias term, we have
\begin{equation}
\label{eq:conditional-inequality}
\begin{aligned}
\E[\lVert h_{n+1} -h^* \lVert_X^2|\mathcal{F}_n] &\leq \lVert h_n - h^* \lVert_X^2 - 2t_n ( \nabla j(h_n)+b_n, h_n-h^*)_X + t_n^2 \E[\lVert g_n\rVert_X^2|\mathcal{F}_n]\\
&\leq (1+2t_nK_n)\lVert h_n - h^* \lVert_X^2 -2t_n(j(h_n)-j(h^*)) + t_n^2 M.
\end{aligned}
\end{equation}
Lemma~\ref{lemma:convegence_nonnegative_supermartingales} implies that $\{\lVert h_n - h^*\rVert_X \}$ is a.s.~convergent and $\sum_{n=1}^\infty t_n (j(h_n)-j(h^*)) < \infty$ a.s., from which we can conclude that $\liminf_{n \rightarrow \infty} j(h_n) = j(h^*)$ with probability one. To show that in fact $\lim_{n\rightarrow \infty} j(h_n) = j(h^*)$, can use a simpler argument than in \cite[Theorem 3.6]{Geiersbach2019a} since we assumed $X_{ad}$ to be bounded. Indeed, applying expectation \eqref{eq:conditional-inequality} and again using Lemma~\ref{lemma:convegence_nonnegative_supermartingales}, we obtain  that $\sum_{n=1}^\infty t_n \E[j(h_n)-j(h^*)] < \infty$ surely. Convexity of $j$ implies that 
\begin{align*}
    j(h_n)-j(h_{n+1}) &\leq (\nabla j(h_n),h_n-h_{n+1})_X\\
    & \leq \lVert \nabla j(h_n) \rVert_X \lVert h_n-h_{n+1}\rVert_X\\
    & =\lVert \nabla j(h_n) \rVert_X \lVert \pi_{X_{ad}}(h_n)-\pi_{X_{ad}}(h_{n}-t_n g_n)\rVert_X\\
    & \leq \lVert \nabla j(h_n) \rVert_X \lVert t_n g_n\rVert_X.
\end{align*}
Since $X_{ad}$ is bounded, so is $\{h_n\}$, and so there exists a $\tilde{M}>0$ such that $ \lVert \nabla j(h_n) \rVert_X \leq \tilde{M}$ for all $n$. After applying the conditional expectation, we have that 
\[
j(h_n)-\E[j(h_{n+1})|\mathcal{F}_n] \leq t_n \tilde{M} \E[\lVert g_n\rVert_X|\mathcal{F}_n].
\] 
From Jensen's inequality, we see that 
\[
(\E[\lVert g_n\rVert_X|\mathcal{F}_n])^2 \leq \E[\lVert g_n\rVert_X^2|\mathcal{F}_n] \leq M,
\]
from which we can conclude that $j(h_n)-\E[j(h_{n+1})|\mathcal{F}_n] \leq t_n \tilde{M} M.$
Now, we can apply Proposition~\ref{prop:technicalproposition} to conclude that  $\lim_{n\rightarrow \infty} j(h_n) = j(h^*) = 0$, which was the first claim.

For the second claim, we observe an arbitrary trajectory of the random sequence $\{h_n\}$. Since $\{\lVert h_n -h^*\rVert_X^2\}$ is convergent, it is also bounded. In particular, there exists a weak accumulation point $\bar{h} \in X_{ad}$ of the sequence $\{h_n\}$. Let $\{h_{n_k}\}$ be a subsequence such that $h_{n_k}\rightharpoonup_X \bar{h}$. By weak lower semicontinuity of $j$ (which follows from the continuity and convexity of $j$ in $X$), we have
\[
j(\bar{h})\leq \lim_{k\rightarrow \infty} j(h_{n_k}) = j(h^*),
\]
where equality follows by the first part of this proof. Since $h^*$ is the (unique) minimizer, it follows that $j(h^*) = j(\bar{h})$, from which we can conclude that $h^* = \bar{h}$. The accumulation point being unique gives in fact $h_{n} \rightharpoonup_X h^*$ with probability one.

The third claim follows now directly from \eqref{eq:strong-convexity-2}, namely that $\lVert h_n - h^*\rVert_H^2 \rightarrow 0$ a.s.~as $n \rightarrow \infty$.
\end{proof}

\subsection{Discussion}
\label{sec:Discussion}
The above result is stronger than it may seem at first glance. While $j$ is strongly convex with respect to $H$, it is only convex with respect to $X$. On the other hand, the $H$-strong convexity makes $j$ strictly convex in $X$, from which we can conclude that a unique solution exists. We proved that, at least with respect to the weaker norm $H$, we can expect (almost sure) strong convergence of method to this unique minimizer.
We note that the dimension of the underlying random vector $\xi$ appears to be immaterial in the original result from \cite[Theorem 3.6]{Geiersbach2019a}. 

Theorem~\ref{lemma:convergenceofPSG} provides the argument for almost sure convergence of the projected stochastic gradient method. It is natural to ask whether convergence rates (in the mean square) can be derived as in \cite{Geiersbach2020}. Interestingly, because of the two-norm discrepancy, one cannot obtain the expected convergence rates for strongly convex problems. Let us investigate this further.

Note that strong convexity \eqref{eq:strong-convexity} of $j$ implies
\begin{equation}
\label{eq:strong-convexity-3}
 D j(h_n)[h_n-h^*]-D j(h^*)[ h_n - h^*] \geq {c_E} \lVert h_n - h^*\rVert_H^2
\end{equation}
for all $n$. Picking up from the estimate in \eqref{eq:first-steps-convergence}, we write 
\begin{align*}
&\E[\lVert h_{n+1} -h^* \lVert_X^2|\mathcal{F}_n]\\ 
& \qquad \leq  \lVert h_n - h^* \lVert_X^2 - 2t_n ( \E[g_n|\mathcal{F}_n], h_n-h^*)_X + t_n^2 \E[ \lVert  g_n\rVert_X^2|\mathcal{F}_n]\\
&\qquad\leq \lVert h_n - h^* \lVert_X^2 - 2t_n ( \nabla j(h_n)+b_n, h_n-h^*)_X + t_n^2 \E[\lVert  g_n\rVert_X^2|\mathcal{F}_n]\\
&\qquad\leq \lVert h_n - h^* \lVert_X^2 - 2t_n ( \nabla j(h_n) -\nabla j(h^*) +b_n, h_n-h^*)_X + t_n^2 \E[\lVert  g_n\rVert_X^2|\mathcal{F}_n]\\
&\qquad\leq (1+2t_nK_n)  \lVert h_n - h^* \lVert_X^2 
-2t_nc_E \lVert h_n - h^* \lVert_H^2 +2t_n K_n + t_n^2 M.
\end{align*}
Due to the mixture of norms in the final line, we fail to produce the recursion necessary to prove the usual rate for strongly convex functions. If we tried to do the above computations, but in the $H$ norm, we would fail because 
\[
(\nabla j(h_n), h_n-h^*)_H \neq Dj(h_n)[h_n-h^*],
\]
and $\nabla j(h_n)$ is the Riesz representative with respect to $X$, not $H$. 

On the other hand, in the numerical section, we will observe convergence rates that fit the theory for strongly convex functions. Once discretized, the underlying spaces are finite-dimensional, where all norms are equivalent. In the finite-dimensional case with norm $\lVert \cdot \rVert$, if step sizes are chosen such that $t_n = \theta/n$ with $\theta > 1/(2c_E)$, we can expect in the unbiased case (see \cite{Nemirovski2009}):
\[ \E[\lVert h_n - h^* \rVert] \leq \sqrt{\frac{\rho}{n}}\] with $\rho = \max \{\lVert h_1-h^*\rVert^2, \theta^2M(2c_E\theta-1)^{-1} \}$. Moreover, if $h^*$ satisfies $\nabla j(h^*) = 0 $ (i.e., $h^*$ is an interior point of $j$ in finite dimensions), $C_S$-Lipschitz continuity of $j$ gives the following rate of convergence for function values:
\begin{equation*}\label{eq:objectiveerrorbounds}
\E[j(h_n)-j(h^*)] \leq \frac{C_S \rho}{2n}. 
\end{equation*}
We note that in the case where $\nabla j(h^*) = 0$, $C_S$-Lipschitz continuity of $\nabla j$ allows us to obtain a convergence rate of the expected norm of the gradient, since by Lipschitz continuity of $\nabla j$,
\begin{equation}
\label{eq:convergence-rate-gradient}
\E[\lVert \nabla j(h_n) \rVert] \leq C_S\E[\lVert h_n - h^* \rVert] \leq C_S\sqrt{\frac{\rho}{n+\nu}}.
\end{equation}
The convergence rate \eqref{eq:convergence-rate-gradient} will indeed be observed in the numerical simulation, even though we cannot show this in the appropriate function space.

As a final comment, we remark that the assumptions made on measurability in Assumption \ref{assumption:gradient} are not too strong, as shown in the following lemma from \cite{Geiersbach2023}. We recall our assumption that the image space $\Xi$ of the random vector $\xi$ is a complete separable metric space.
\begin{lemma}
\label{lem:measurability}
Suppose $X$ is also assumed to be separable and $\{\mathcal{F}_n\}$ is the natural filtration generated by the stochastic process $\{ \xi_n\}.$ Suppose $G\colon X \times \Xi \rightarrow X$ and $\nabla j\colon X \rightarrow X$ are continuous with respect to the $X$ norm in $X_{ad} \times \Xi$ and $X_{ad}$, respectively. Then $h_n$, defined by the recursion \eqref{eq:SGM-recursion}, as well as the functions $b_n$ and $w_n$ respectively, are adapted to $\mathcal{F}_n$ for all $n.$
\end{lemma}
\section{Numerical results}
\label{sec:numerix}
For our numerical setting, we assume that both, the 
interior and the exterior boundary, are starlike and 
use polar coordinates to parameterize them. The associated
exterior radial function is represented by the finite 
Fourier series
\begin{equation}\label{eq:radial}
  r_\Gamma(\theta) = a_{0,\Gamma} + \sum_{\ell=1}^N a_{-\ell,\Gamma}\sin(\ell\theta)
    + a_{\ell,\Gamma}\cos(\ell\theta), \quad \theta\in[0,2\pi],
\end{equation}
and likewise the interior one by
\[
  r_\Sigma(\theta,\omega) - \overline{r}_\Sigma(\theta) = \xi_0(\omega) 
    + \sum_{\ell=1}^N \xi_{-\ell}(\omega)\sin(\ell\theta)
    + \xi_\ell(\omega) \cos(\ell\theta), \quad \theta\in[0,2\pi].
\]
Here, $\overline{r}_\Sigma(\theta)$ is chosen as the
radial function which describes the ellipse with semi-axes 
$0.4$ and $0.2$, while the random variables $\xi_\ell(\omega)
\in\mathcal{U}([-0.5,0.5])$ are uniformly distributed and
independent. We thus have $\E[r_\Sigma(\theta)] = 
\overline{r}_\Sigma(\theta)$.

For our numerical experiments, we employ 17 degrees of freedom 
in \eqref{eq:radial}, which corresponds to $N=8$. Due to the use 
of finite dimensional Fourier series, both boundaries are always 
$C^\infty$-smooth and of bounded curvature provided that the radial 
functions are uniformly bounded from above and below. Especially, 
the Riesz projection \eqref{eq:grad_projection} of the discretized 
gradient is just the identity as the gradient is a member of 
$X$. Also the realization of the projection of the exterior 
boundaries onto the class of convex boundaries becomes obsolete 
as the exterior boundary is always convex during the runs 
of the stochastic gradient method.

The $H^{1/2}$-energy norm of the shape gradient is 
realized by applying an appropriate scaling of 
its Fourier coefficients. The initial guess for the 
exterior boundary is a circle of radius $0.75$, which is 
centered in the origin (compare Figure~\ref{fig:outcome} 
top left). It was not neceessary to impose constraints 
on the parameterization of the outer boundary, as it 
is of bounded curvature since the radial function 
consists only of a few terms. Moreover, we never 
observed difficulties in the numerical simulations which 
is in line with the observations made in \cite{EP06} 
that the optimization problem under consideration 
is convex in the present setting despite of the 
non-convex boundaries. Note that all the details 
of the implementation, which is based on a boundary 
element method, can be found therein, too.

\begin{figure}
\begin{center}
\includegraphics[width=0.4\textwidth]{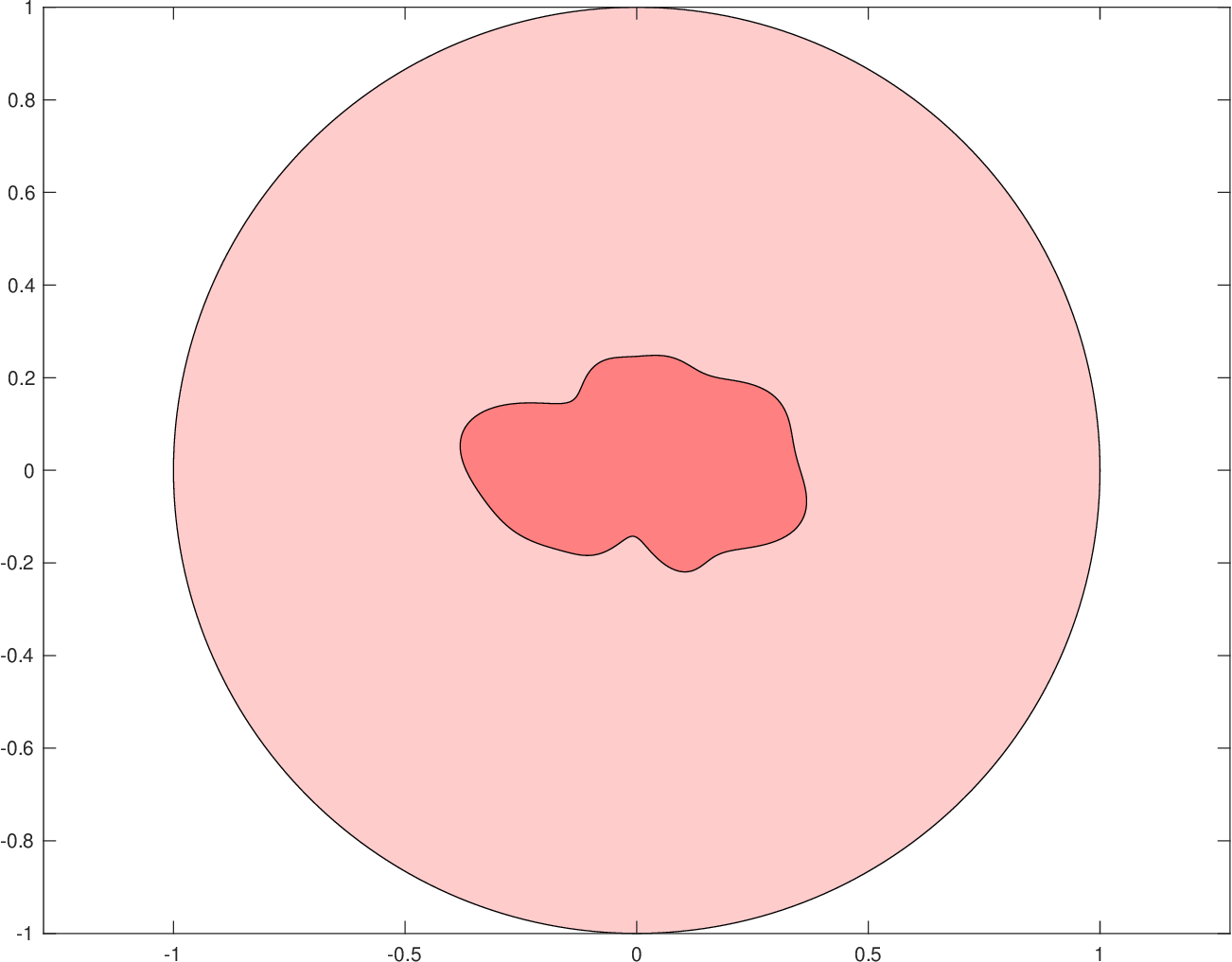}
\includegraphics[width=0.4\textwidth]{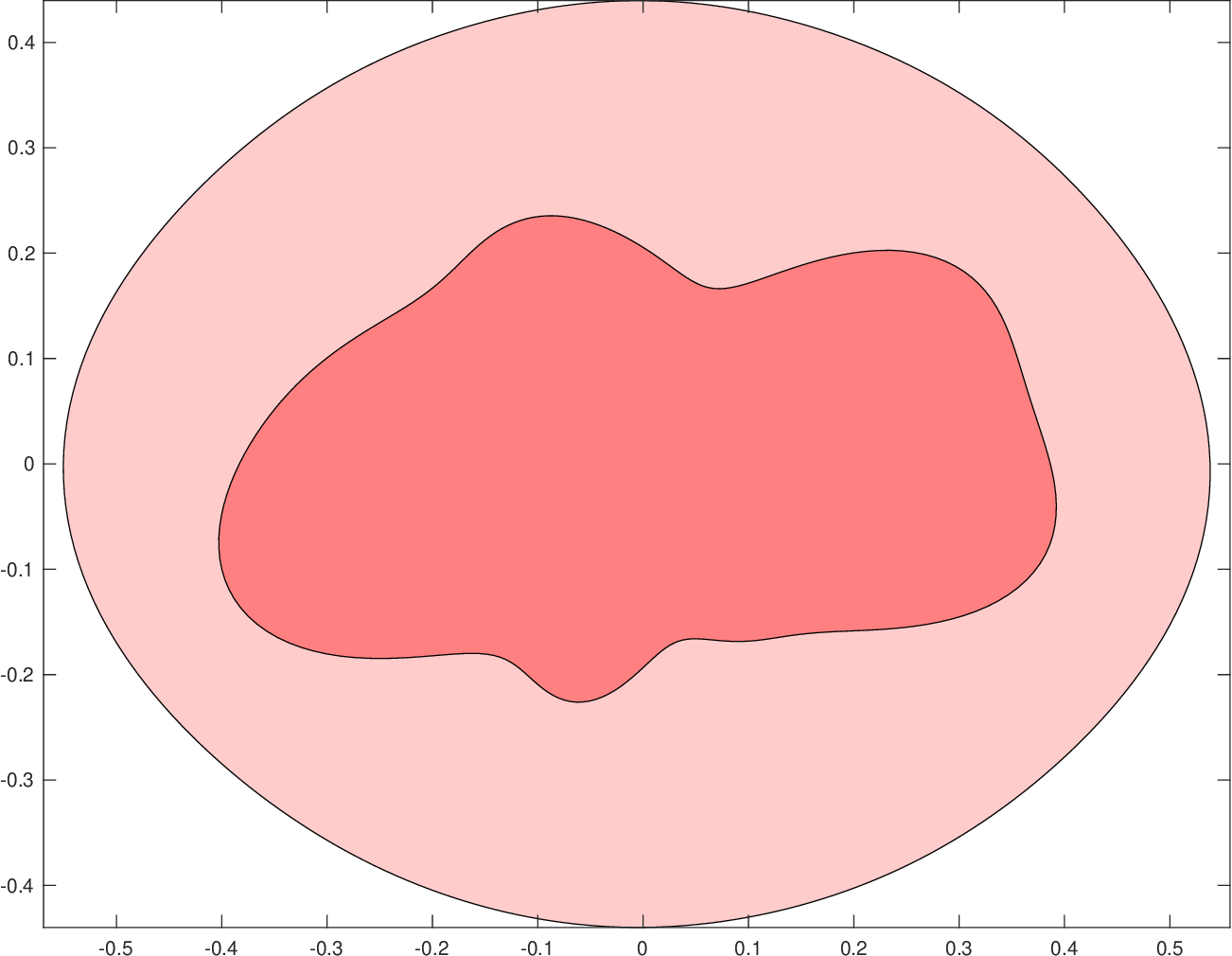}
\includegraphics[width=0.4\textwidth]{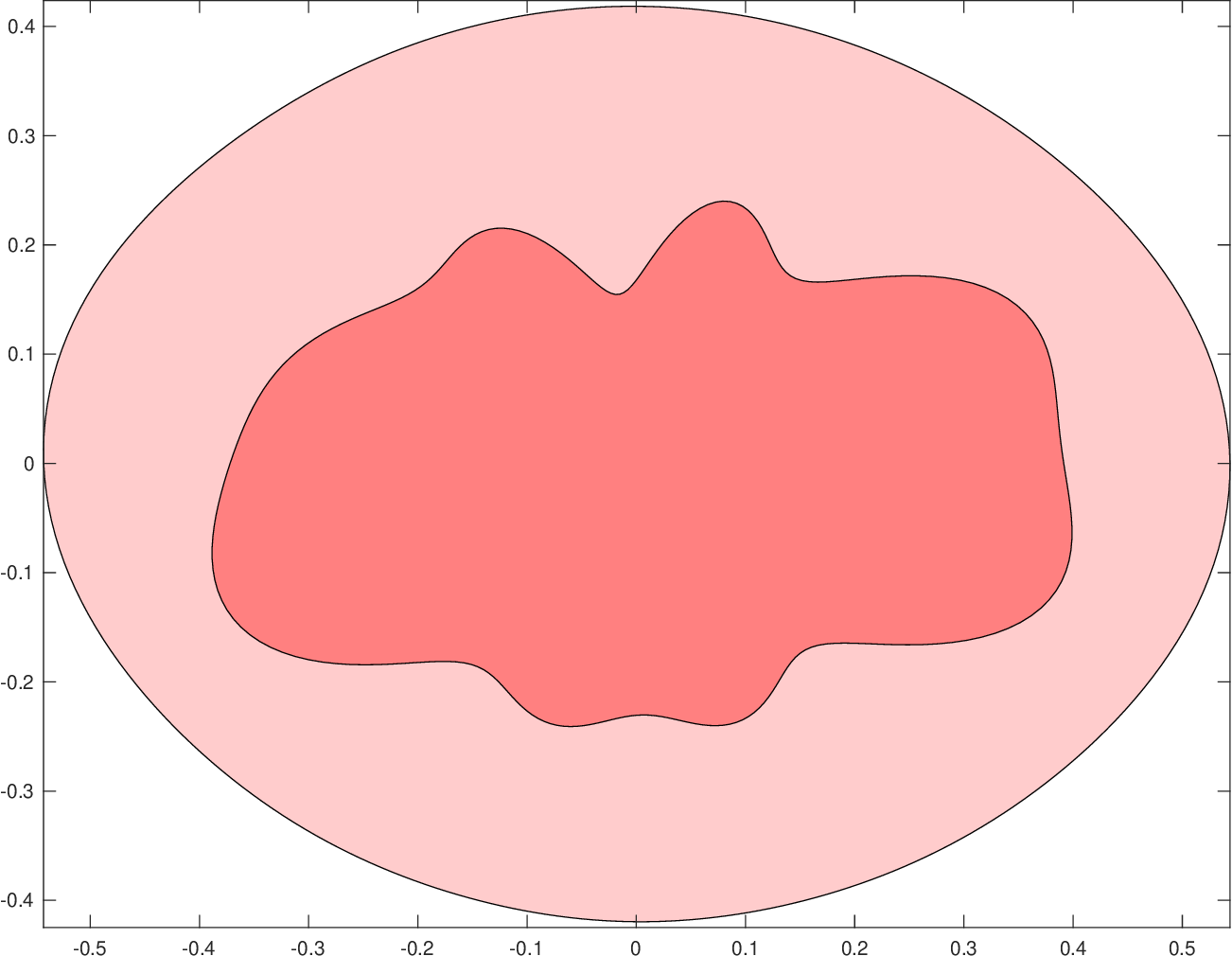}
\includegraphics[width=0.4\textwidth]{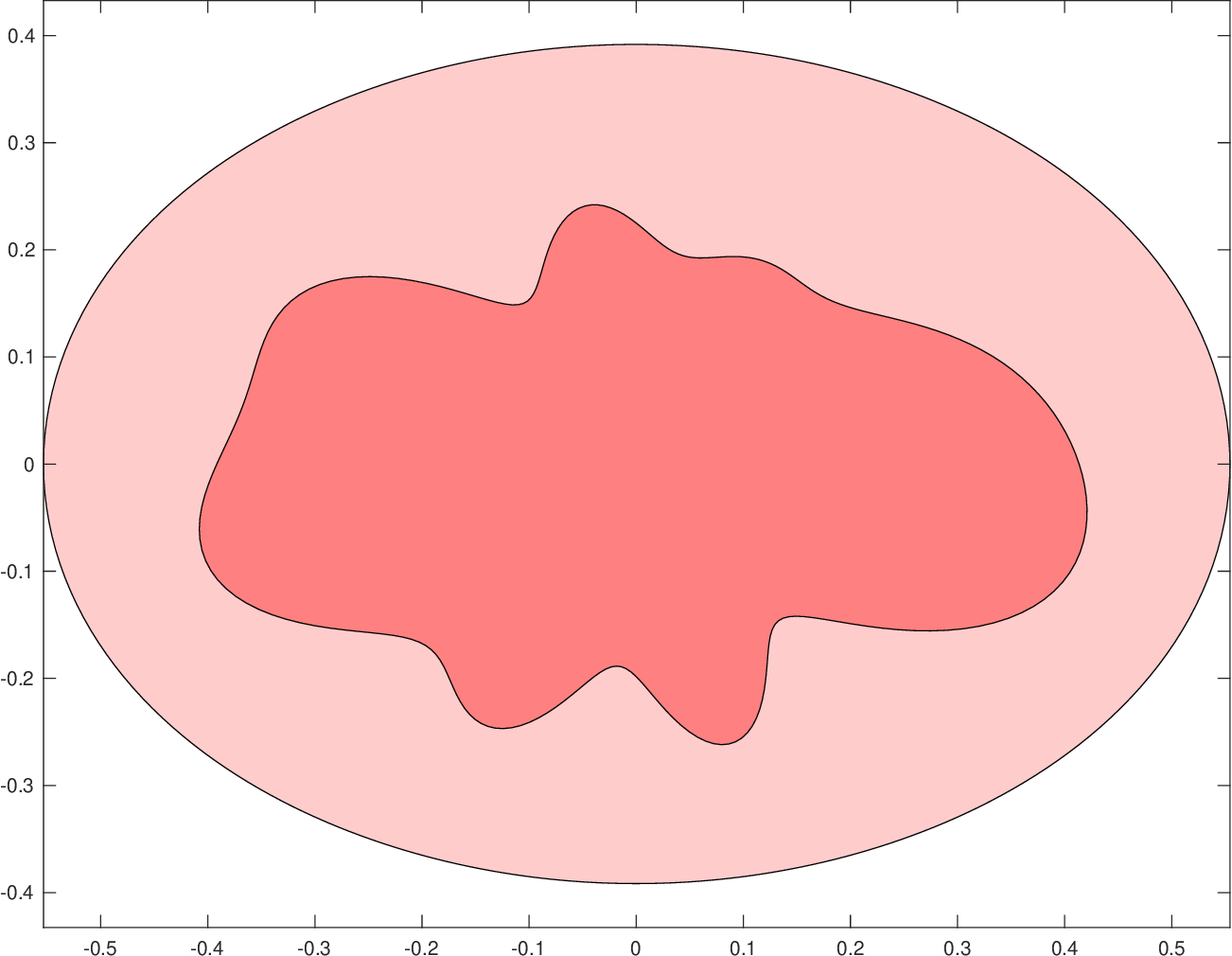}
\caption{\label{fig:outcome}The initial (circular) exterior
boundary (top left), the exterior boundary after 10 iterations 
(top right) after 20 iterations (bottom left), and at after 1000 
iterations (bottom right). The interior boundaries represent 
different random samples.}
\end{center}
\end{figure}

We apply $K$ steps of the stochastic gradient 
method for different numbers of $K$, where the 
step size $t_k$ is in any case chosen in 
accordance with $t_k = \frac{1}{400 k}$. The factor 
$\frac{1}{400}$ is found to be necessary in order to
avoid degeneration of the underlying domains during 
the course of iteration. We observe quite a fast 
convergence of stochastic gradient method towards 
the final ellipse-like outer boundary. After already 
10 iterations, we get the result found in the top right 
plot of Figure~\ref{fig:outcome}, while after 20 iterations
we get we get the result found in the bottom left 
plot of Figure~\ref{fig:outcome}. The boundary 
computed after $K=10\,000$ iterations is found in 
the bottom right plot of Figure~\ref{fig:outcome}.
The interior boundaries seen in Figure~\ref{fig:outcome}
represent different draws of the random interior
boundary.

\begin{figure}
\begin{center}
\includegraphics[width=0.45\textwidth]{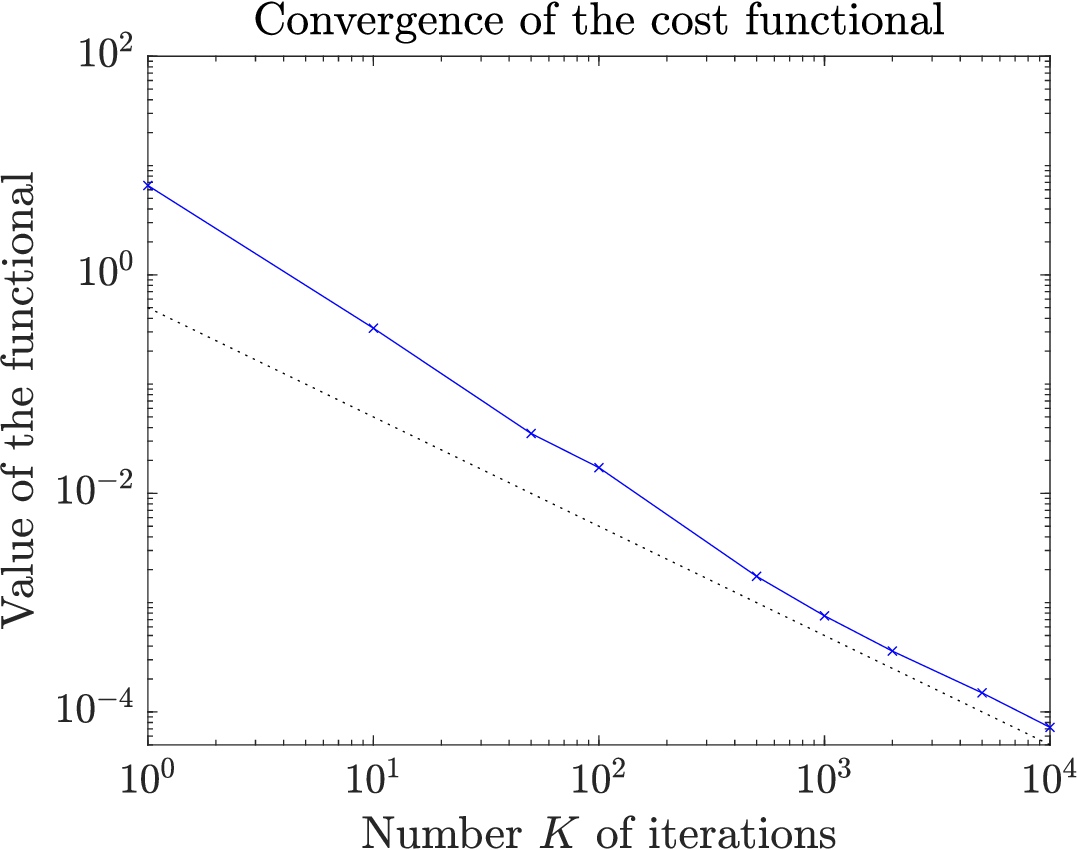}
\includegraphics[width=0.45\textwidth]{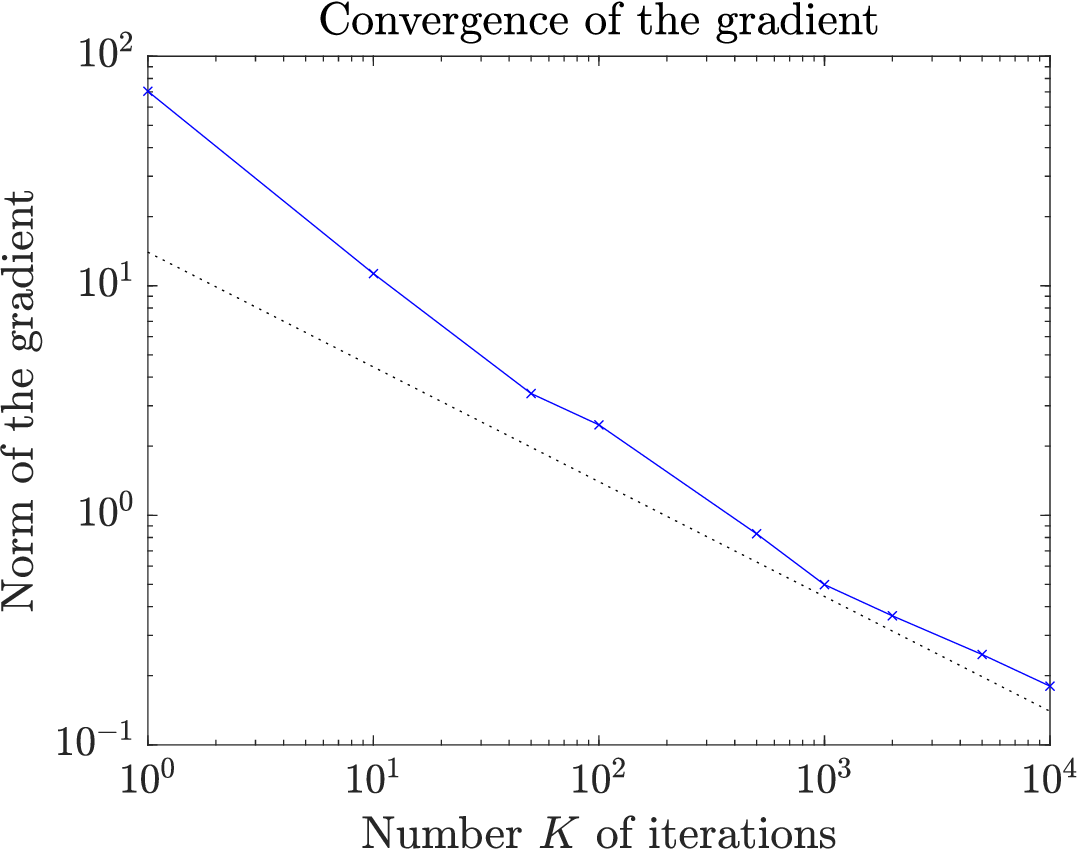}
\caption{\label{fig:convergence} Convergence of the 
stochastic gradient method with respect to the number $K$
of samples. The values of the cost functional are seen on 
the left, the norm of the gradient is seen on the right.
We observe the rate of convergence $K^{-1}$ for the cost 
functional and $K^{-1/2}$ for the gradient, indicated
by the dotted black lines.}
\end{center}
\end{figure}

In Figure~\ref{fig:convergence}, we plot 
the error between the mean energy functional 
and its minimizer as well as the norm of the 
respective shape gradient for the $K$-th 
iterate versus the number $K$ of iterations 
of the stochastic gradient method. Both expectations
are computed by a quasi-Monte Carlo method using 
1000 samples. Moreover, each particular data point 
reflects the mean of three runs of the stochastic 
gradient method.

One can read from the right plot in Figure~\ref{fig:convergence} 
that the norm of the initial mean gradient is approximately 
70, while after $K = 10\,000$ iterations the norm 
of the mean gradient lies between 0.010 and 0.020, 
depending on the specific run. The cost functional
converges towards the value $E_{\min}\approx 31.856$,
which has been computed by using $K=20\,000$ samples
in the stochastic gradient method, compare the left 
plot in Figure~\ref{fig:convergence}. We observe the 
rate $K^{-1}$ of convergence for the cost functional 
while it is $K^{-1/2}$ for the norm of the gradient. 
These rates are indicated by the dotted lines in 
Figure~\ref{fig:convergence}. Indeed, the rate of
convergence seems to be a bit faster for the first 
few samples in the beginning.

\section{Conclusion}
\label{sec:conclusio}
In the present article, we developed the convergence theory of 
the stochastic gradient method in case of a problem which exhibits
the two-norm discrepency. The two-norm discrepency is a well-known
phenomenon in the optimal control of partial differential equations. 
We considered exemplarily Bernoulli's free boundary problem with 
a random interior boundary which can be seen as a fruit fly of 
a shape optimization problem under uncertainty. We have proven 
the strong convexity of the underlying shape optimization problem 
with respect to the $H^{1/2}$-norm, being weaker than the 
$C^{3,2\alpha}$-regularity required to ensure differentiability. 
Numerical results validate our theoretical findings.

\subsection*{Acknowledgement}
This research has been in part performed while H.H.\ was visiting the 
Laboratory of Mathematics and its Applications of PAU -- UMR CNRS 5142.
The hospitality and the support are gratefully acknowledged.
\bibliography{references}
\end{document}